\documentclass[12pt,reqno]{amsart}

\usepackage{latexsym}
\usepackage{graphicx}

\newtheorem{theorem}{{\sc Theorem}}[section]

\newtheorem{lemma}[theorem]{{\sc Lemma}}
\newtheorem{prop}[theorem]{{\sc Proposition}}

\theoremstyle{remark}
\newtheorem{remark}[theorem]{{\sc Remark}}

\def\f{\mathfrak }
\def\b{\mathbb }

\def\rr{\b R}

\def\nn{\b N}
\def\cc{\b C}

\def\pp{\b P}

\def\erw{\b E}

\def\phi{\varphi }

\def\calf{{\mathcal F}}
\def\cala{{\mathcal A}}
\def\calb{{\mathcal B}}

\def\cald{{\mathcal D}}

\def\calw{{\mathcal W}}

\def\frk{{\f k}}

\def\on{\operatorname}

\def\tra{^{\prime}}

\begin{document}

\title[Multivariate CLT for traces of powers]{Stein's method and the multivariate CLT for traces of  powers on the classical compact groups}

\author{Christian D\"obler}
\address{Ruhr-Universit\"at Bochum, Fakult\"at f\"ur Mathematik, NA 3/68, D-44780 Bochum, Germany.} 
\email{christian.doebler@rub.de}

\author{Michael Stolz}\thanks{Both authors have been supported by Deutsche Forschungsgemeinschaft via SFB-TR 12.
}
\address{Ruhr-Universit\"at Bochum, Fakult\"at f\"ur Mathematik, NA 3/69, D-44780 Bochum, Germany.}
\email{michael.stolz@ruhr-uni-bochum.de}

\begin{abstract}
Let $M_n$ be a random element of the unitary, special orthogonal, or unitary symplectic groups, distributed 
according to Haar measure. By a classical result of Diaconis and Shahshahani, for large matrix size $n$,
the vector $ (\on{Tr}(M_n), \on{Tr}(M_n^2), \ldots, \on{Tr}(M_n^d))$ tends to a vector of independent (real or complex) Gaussian random variables. Recently, Jason Fulman has demonstrated that for a single power $j$ (which may grow with $n$), 
a speed of convergence result may be obtained via Stein's method of exchangeable pairs. In this note, we extend 
Fulman's result to the multivariate central limit theorem for the full vector of traces of powers.
\end{abstract}

\maketitle

One aspect of random matrix theory concerns random elements of compact Lie groups. A classical result, due to
Diaconis and Shahshahani \cite{DiaSh94}, is as follows: Let $M_n$ be an element of $\on{U}_n, \on{O}_n$, or $\on{USp}_{2n}$, distributed according to Haar measure. Then, as $n \to \infty$, the vector
$$ (\on{Tr}(M_n), \on{Tr}(M_n^2), \ldots, \on{Tr}(M_n^d))$$ converges weakly to a vector of independent, (real or complex) Gaussian random variables. The original proof deduced this from exact moment formulae, valid for $n$ sufficiently large (see also \cite{Sto05}). Different versions of the moment computations, also taking care of $\on{SO}_n$, have been
proposed in \cite{PaVa04} and \cite{HuRud03}.

Subsequently, the speed of convergence in the univariate version of this result was studied by Charles Stein \cite{Ste95},
who proved that in the orthogonal case the error decreases faster than any power of the dimension, and Kurt Johansson \cite{Joh97}, who obtained
exponential convergence. While Johansson's approach had its roots in Szeg\"o's limit theorem for Toeplitz determinants, Stein used the ``exchangeable pairs'' version of a set of techniques that he had been developing since the early 1970s (see \cite{Ste72}) and that nowadays is referred to as ``Stein's method''. 
Recently, Jason Fulman {\cite{Ful10} has proposed an approach to the speed of convergence for a single power of $M_n$, based on combining Stein's method of exchangeable pairs with heat kernel techniques. While producing weaker results on the speed than Stein's and Johansson's, his theorems apply to the case that the power $M_n^{d(n)}$ grows with $n$. Furthermore, his techniques seem more likely to be useful in contexts beyond the standard representations of the classical groups.

In this note, we extend Fulman's approach and results to a multivariate setting, making use of recent extensions,
due to Chatterjee, Meckes, Reinert, and R\"ollin \cite{ChaMe08, ReiRoe09, Me09}, of Stein's method of exchangeable pairs 
to cover multivariate normal approximations. This yields, to the best of our knowledge, the first rates of convergence 
result in the multivariate CLT for traces of powers on the classical compact groups, or, in any case, the first one that allows 
for the powers to grow with $n$. Our set-up contains Fulman's as a special case, and we recover, in Wasserstein distance, the rates that were proven by him in Kolmogorov distance (see Remark \ref{fulmanvergleich} below).

We will review Meckes' version
of the multivariate exchangeable pairs method (\cite{Me09}) in Section \ref{Meckes}, stating a complex version of her 
results that will be useful for the unitary case.
This will lead to a slight improvement on Fulman's result even in the one dimensional (single power) case, in that now the convergence of the (complex) trace of a power of a random Haar unitary
to a complex normal distribution can be quantified. 

Fulman's approach consists in constructing 
useful exchangeable pairs from Brownian motion on the connected compact Lie groups $K_n = \on{U}_n, \on{SO}_n,$ or
$\on{USp}_{2n}$. This amounts to studying the heat kernel on $K_n$, and in particular the action 
of the Laplacian on power sum symmetric polynomials, which express products of traces of powers of group elements 
in terms of their eigenvalues. This is reviewed in Section \ref{Rains}. Explicit formulae are due to Rains \cite{Rai97}.
They are rooted in Schur-Weyl duality as explained in \cite{Lev08}. 

In Sections \ref{sec-u}, \ref{sec-so}, and \ref{sec-sp}, we then state and prove rates of convergence results in the multivariate CLT for
the unitary, special orthogonal, and unitary symplectic groups.

\section{Exchangeable pairs and multivariate normal approximation}
\label{Meckes}

The approach of univariate normal approximation by exchangeable pair couplings in Stein's method has a long history 
dating back to the monograph \cite{St86} by Charles Stein in 1986. In order to show that a given real valued random variable $W$ 
is approximately normally distributed, Stein proposes the construction of another random variable $W\tra$ on the same probability space as $W$ such that the pair $(W,W\tra)$ is exchangeable, i.e. satisfies the relation 
$(W,W\tra)\stackrel{\cald}{=}(W\tra,W)$, and such that there exists $\lambda\in ]0,1[$ for which the
linear regression property $\erw[W-W\tra|W]=-\lambda W$ holds. In this situation, the distance between $W$ and a standard normal $Z$ can be bounded in various metrics, including Wasserstein's and Komogorov's.

The range of examples to which this technique could be applied was considerably extended in the work \cite{RiRo97} of Rinott and Rotar who proved normal approximation theorems allowing the linear regression property to be satisfied
only approximately. Specifically, they assumed the existence of a ``small'', random quantity $R$ such that 
$\erw[W\tra - W|W]=-\lambda W +R$. 

In \cite{ChaMe08}, Chatterjee and Meckes proposed a version of exchangeable pairs for multivariate normal
aproximation. For a given random vector $W=(W_1,\ldots,W_d)^T$ they assume the existence of another random vector 
$W\tra =(W_1\tra,\ldots,W_d\tra)^T$ such that $W \stackrel{\cald}{=} W\tra$ and of a constant $\lambda\in\rr$ such that 
$\erw[W\tra - W|W]=\lambda W$. In \cite{ReiRoe09} Reinert and R\"{o}llin investigated the more general linear regression property $\erw[W\tra -W|W]=-\Lambda W+R$, where now $\Lambda$ is an invertible non-random $d\times d$ matrix and $R=(R_1,\ldots,R_d)^T$ is a small remainder term. However, in contrast to Chatterjee and Meckes, Reinert and R\"{o}llin need the full strength of the exchangeability of the vector $(W,W\tra)$. Finally, in \cite{Me09}, Elizabeth Meckes reconciled the two approaches, allowing for the more general linear regression property from \cite{ReiRoe09} and using sharper coordinate-free bounds on the solution to the Stein equation suggested by those from \cite{ChaMe08}. 

Both \cite{ChaMe08} and \cite{Me09} contain an ``infinitesimal version of Stein's method of exchangeable pairs'', i.e., 
they provide error bounds for multivariate normal approximations in the case that for each $t>0$ there is an exchangeable pair $(W,W_t)$ and that some further limiting properties hold as $t\to0$. It is such an infinitesimal version that
will be applied in what follows.

The abstract multivariate normal approximation theorem provided by Meckes will be sufficient for the special orthogonal and symplectic cases of the present note, and the unitary case needs only a slight extension via basic linear algebra. So it
is not really necessary to explain how these approximation theorems relate to the fundamentals of Stein's method
(see \cite{ChSh05} for a readable account). But it is crucial to understand that there are fundamental differences
between the univariate and multivariate cases. So a few remarks are in order. The starting point of any version
of Stein's method of normal approximation in the univariate case is the observation that a random variable $W$ is standard
normal if, and only if, for all $f$ from a suitable class of piecewise $\on{C}^1$ functions one has 
$$ \erw[ f\tra(W) - W f(W)] = 0.$$ The quest for an analogous first order characterization of the multivariate
normal distribution has proven unsuccessful, so Chatterjee and Meckes work with the following second-order
substitute (see \cite[Lemma 1]{Me09}): Let $\Sigma$ be a positive semi-definite $d\times d$-matrix. A $d$-dimensional random vector $X$ has the distribution 
$N(0,\Sigma)$ if, and only if, for all $f\in \on{C}_c^2(\rr^d)$ the identity
$$ \erw\left[\left\langle\on{Hess}f(X),\Sigma\right\rangle_{\rm HS}- \left\langle X,\nabla f(X)\right\rangle\right]=0$$
holds (where HS stands for Hilbert-Schmidt, see below). Among the consequences of this is that the multivariate
approximation theorems are phrased in Wasserstein distance (see below) rather than Kolmogorov's distance
$$ \sup_{t \in \rr}\ \big|\mu( ]-\infty, t]) - \nu( ]-\infty, t])\big|$$ for probability measures $\mu, \nu$ on the real line, i.e., the distance concept in which Fulman's univariate theorems are cast.

For a vector $x\in\rr^d$ let $\|x\|_2$ denote its euclidean norm induced by the standard scalar product on $\rr^d$ that will be denoted by $\langle\cdot,\cdot\rangle$. For $A, B \in \rr^{d \times k}$ let 
$\langle A,B\rangle_{\rm HS}:=\on{Tr}(A^TB)=\on{Tr}(B^TA)=\on{Tr}(AB^T)=\sum_{i=1}^d\sum_{j=1}^k a_{ij}b_{ij}$ be the usual Hilbert-Schmidt scalar product on $\rr^{d\times k}$ and denote by $\|\cdot\|_{\rm HS}$ the corresponding norm.
For random matrices $M_n, M \in \rr^{k \times d}$, 
defined on a common probability space $(\Omega, \cala, \pp)$, we will say that $M_n$ converges to $M$ 
in $\on{L}^1(\| \cdot \|_{\rm HS})$
if $\| M_n - M\|_{\rm HS}$ converges to $0$ in $\on{L}^1(\pp).$

For $A \in \rr^{d \times d}$ let $\| A\|_{\rm op}$ denote the operator norm induced by the euclidean norm, i.e.,
$\| A\|_{\rm op} = \sup\{ \|Ax\|_2:\ \|x\|_2 = 1\}.$
More generally, for a $k$-multilinear form $\psi:(\rr^d)^k\rightarrow\rr$ define the operator norm
$$ \|\psi\|_{\rm op}:=\sup\left\{\psi(u_1,\ldots,u_k)\,:\, u_j\in\rr^d,\, \|u_j\|_2=1,\, j=1,\ldots,k\,\right\}.$$

For a function $h:\rr^d\rightarrow\rr$ define its minimum Lipschitz constant $M_1(h)$ by

\[M_1(h):=\sup_{x\not=y}\frac{|h(x)-h(y)|}{\|x-y\|_2}\in[0,\infty].\]

If $h$ is differentiable, then $M_1(h)=\sup_{x\in\rr^d}\|Dh(x)\|_{\rm op}$. 
More generally, for $k\geq1$ and a $(k-1)$-times differentiable $h:\rr^d\rightarrow\rr$ let

\[M_k(h):=\sup_{x\not=y}\frac{\|D^{k-1}h(x)-D^{k-1}h(y)\|_{\rm op}}{\|x-y\|_2}\,,\]

viewing the $(k-1)$-th derivative of $h$ at any point as a $(k-1)$-multilinear form.
Then, if $h$ is actually $k$-times differentiable, we have $M_k(h)=\sup_{x\in\rr^d}\|D^kh(x)\|_{\rm op}$. 
Having in mind this identity, we define  $M_0(h):=\|h\|_\infty$. 

Finally, recall the Wasserstein distance for probability distributions $\mu$, $\nu$ on $(\rr^d,\calb^d)$. It is defined by 

$$d_\calw(\mu,\nu):=\sup\left\{|\int h d\mu-\int h d\nu|\,:\, h:\rr^d\rightarrow\rr\text{ and } M_1(h)\leq1\right\}.$$

Now we are in a position to state the abstract multivariate normal approximation theorems that will be applied in this note, starting with the
real version, taken from \cite[Thm.~4]{Me09}. $Z = (Z_1, \ldots, Z_d)^T$ denotes a standard $d$-dimensional 
normal random vector, $\Sigma \in \rr^{d \times d}$ a positive semi-definite matrix and $Z_{\Sigma} := \Sigma^{1/2} Z$ with
distribution $\on{N}(0, \Sigma)$.   

\begin{prop}\label{Meckes-r}
Let $W, W_t\ (t > 0)$ be $\rr^d$-valued $\on{L}^2(\pp)$ random vectors on the same probability space
$(\Omega, \cala, \pp)$ such that for any $t>0$ the pair $(W, W_t)$ is exchangeable. 
Suppose there exist an invertible non-random matrix $\Lambda$, a positive semi-definite matrix $\Sigma$, a random vector $R=(R_1,\ldots,R_d)^T$, a random $d\times d$-matrix $S$, a sub-$\sigma$-field $\calf$ of $\cala$ such that $W$ 
is measurable w.r.t.\ $\calf$ and a non-vanishing deterministic function $s:\ ]0, \infty[ \rightarrow\rr$ such that the following three conditions are satisfied:
\begin{enumerate}
\item[(i)] $\frac{1}{s(t)}\erw[W_t-W|\calf]\stackrel{t\to0}{\longrightarrow}-\Lambda W+R$ in $\on{L}^1(\pp)$.
\item[(ii)] $\frac{1}{s(t)}\erw[(W_t-W)(W_t-W)^T|\calf]\stackrel{t\to0}{\longrightarrow} 2\Lambda\Sigma + S$ 
in $\on{L}^1(\|\cdot\|_{\rm HS})$.
\item[(iii)] For each $\epsilon>0$,
$$\lim_{t\to0}\frac{1}{s(t)} \erw\left[\|W_t-W\|_2^2\ 1_{\{\|W_t-W\|_2^2>\epsilon\}}\right]=0.$$
\end{enumerate}
Then
\begin{enumerate}
\item[(a)] For each $h\in \on{C}^2(\rr^d)$,
$$|\erw[h(W)]-\erw[h(Z_\Sigma)]|\leq \|\Lambda^{-1}\|_{\rm op}\left(M_1(h) 
\erw[\|R\|_2]+\frac{\sqrt{d}}{4}M_2(h)\ \erw[\|S\|_{\rm HS}]\right).$$

\item[(b)] If $\Sigma$ is actually positive definite, then

$$ d_\calw(W,Z_\Sigma)\leq \|\Lambda^{-1}\|_{\rm op}
\left(\erw[\|R\|_2]+\frac{1}{\sqrt{2\pi}}\|\Sigma^{-1/2}\|_{\rm op}\ \erw[\|S\|_{\rm HS}]\right).$$

\end{enumerate}
\end{prop}

\noindent
\begin{remark} \label{remark4}
In applications it is often easier to verify the following stronger condition  in the place
of  (iii) of Proposition \ref{Meckes-r}:
\begin{itemize}
\item[$(iii)^{\prime}$] \hspace{10em} $\lim_{t\to0}\frac{1}{s(t)}\erw\left[||W_t-W||_2^3\right]=0.$
\end{itemize}
\end{remark}

Now we turn to a version of Proposition \ref{Meckes-r} for complex random vectors, which will be needed for the case of the unitary group. $\| \cdot \|_{\rm op}$ and $\| \cdot \|_{\rm HS}$ extending in the obvious way, we now denote by 
$Z = (Z_1, \ldots, Z_d)$
a $d$-dimensional {\it complex} standard normal random vector, i.e., there
are iid $\on{N}(0, 1/2)$ distributed real random variables $X_1, Y_1, \ldots, X_d, Y_d$ such that $Z_j = X_j + i Y_j$ for all $j = 1, \ldots, d$.

\begin{prop}
\label{Meckes-c}
Let $W, W_t\ (t > 0)$ be $\cc^d$-valued $\on{L}^2(\pp)$ random vectors on the same probability space
$(\Omega, \cala, \pp)$ such that for any $t>0$ the pair $(W, W_t)$ is exchangeable. Suppose that there exist non-random matrices $\Lambda, \Sigma ·\in \cc^{d \times d}$, $\Lambda$ invertible, $\Sigma$ positive semi-definite, a random vector $R \in \cc^d$, random matrices
$S, T \in \cc^{d \times d}$, a sub-$\sigma$-field $\calf$ of $\cala$ such that $W$ is 
measurable w.r.t.\ $\calf$, and a non-vanishing deterministic 
function $s:\ ]0, \infty[ \to \rr$ with the following 
properties:
\begin{itemize}
 \item[(i)] $\frac{1}{s(t)} \erw[ W_t - W| \calf] \stackrel{t \to 0}{\longrightarrow}
- \Lambda W + R$ in $\on{L}^1(\pp).$ 
\item[(ii)] $\frac{1}{s(t)} \erw\left[ (W_t - W) (W_t - W)^* | \calf \right] 
\stackrel{t \to 0}{\longrightarrow} 2 \Lambda \Sigma + S $ in $\on{L}^1(\| \cdot\|_{\rm HS}).$
\item[(iii)] $\frac{1}{s(t)} \erw\left[ (W_t - W) (W_t - W)^T | \calf \right] 
\stackrel{t \to 0}{\longrightarrow} T $ in $\on{L}^1(\| \cdot\|_{\rm HS})$.
\item[(iv)] For each $\epsilon > 0$, 
$$ \lim_{t \to 0} \frac{1}{s(t)} \erw\left[ \| W_t - W\|_2^2\ 1_{\{ \| W_t - W\|_2^2 > \epsilon
\}} \right] = 0.$$
\end{itemize}
Then, 
\begin{itemize}
\item[(a)] If $h \in \on{C}^2(\rr^{2d})$, then
$$ | \erw h(W) - \erw h(\Sigma^{1/2} Z) | \le \| \Lambda^{-1}\|_{\rm op} 
\left( M_1(h) \erw \|R\|_2 + \frac{\sqrt{d}}{4} \erw\left[ \|S\|_{\rm HS}
+ \|T\|_{\rm HS}\right] \right).$$
\item[(b)] If $\Sigma$ is actually positive definite, then
$$ d_{\calw}(W, \Sigma^{1/2} Z) \le \| \Lambda^{-1}\|_{\rm op}\left(
\erw\|R\|_2 + \frac{1}{\sqrt{2 \pi}} \| \Sigma^{-1/2}\|_{\rm op}
\erw\left[ \| S\|_{\rm HS} + \| T\|_{\rm HS}\right] \right).$$
\end{itemize}
\end{prop}
~

\noindent
\begin{remark}\label{rem-meckesc}  
~
\begin{itemize}
\item[(i)] Remark \ref{remark4} above also applies to the present condition (iv).
\item[(ii)] If $Y$ is a centered $d$-dimensional complex normal random vector, then there exist 
$A \in \cc^{d \times d}$ and a complex standard normal $ Z \in \cc^d$ such that $Y = A Z$.
In this case, $ \erw[Y Y^T] = A \erw[ZZ^T] A^T = 0,$ since $\erw[ZZ^T] = 0$. This
heuristically motivates condition (iii) above for $W$ close to such a $Y$, if one thinks of $T$
as a small remainder term. For an algebraic rationale see \eqref{reellif}
below.
\end{itemize}
\end{remark}

\begin{proof}[Proof of Proposition \ref{Meckes-c}]
The proof is by reduction to the real version in Proposition \ref{Meckes-r}. To this end, we need some preparation. To each $z = (z_1,\ldots, z_d)^T \in \cc^d$ assign the vector $$z_{\rr} := (\on{Re}(z_1), \on{Im}(z_1), \ldots, \on{Re}(z_d), \on{Im}(z_d))^T \in 
\rr^{2d}.$$ Given $A = (a_{jk}) \in \cc^{d\times d}$ define $A_{\rr} \in \rr^{2d \times 2d}$ uniquely by requiring that
$(Az)_{\rr} = A_{\rr} z_{\rr}$ for all $z \in \cc$. In concrete terms, $A_{\rr}$ consists of $2 \times 2$ blocks, where
the block at position $(j, k)$ is given as 
$$ \left( \begin{array}{rr} \on{Re}(a_{jk}) & - \on{Im}(a_{jk})\\  \on{Im}(a_{jk}) & \on{Re}(a_{jk})\end{array} \right).$$ Observe that
\begin{equation}
\label{hilfreellif}
(AB)_{\rr} = A_{\rr} B_{\rr},\ (A^*)_{\rr} = (A_{\rr})^T,\ \text{\rm and}\ (AA^*)_{\rr} = A_{\rr} (A^*)_{\rr}
= A_{\rr} (A_{\rr})^T,
\end{equation}
 and that for $d=1$ one has to specify whether a scalar is to be interpreted as a 
vector or as a matrix.
Writing $J$ for the Kronecker product $\left(\begin{array}{rr} 1 & 0 \\ 0 & -1\end{array} \right) \otimes \on{I}_d,$ one easily verifies the following identity of real $2d \times 2d$ matrices:
\begin{equation}\label{reellif}
z_{\rr} w_{\rr}^T = \frac12 (z w^*)_{\rr} + \frac12 (z w^T)_{\rr} J.\end{equation}
Now, given the data from Proposition \ref{Meckes-c}, we will show that the pairs $(W_{\rr}, (W_t)_{\rr})$ satisfy
the assumptions of Proposition \ref{Meckes-r} with $\Lambda_{\rr}$ for $\Lambda$, $\Sigma\tra := \frac12
\Sigma_{\rr}$ for $\Sigma$, $R_{\rr}$ for $R$ and $\frac12(S_{\rr} + T_{\rr}J)$ for $S$. Clearly  
$(W_{\rr}, (W_t)_{\rr})$ is exchangeable and $\Lambda_{\rr}$ is invertible. By \eqref{hilfreellif}, $\Sigma\tra$ is symmetric and non-negative definite.
To verify assumption (i), observe that 
$$ \frac{1}{s(t)} \erw[ (W_t -W)_{\rr} |\calf] = \frac{1}{s(t)} \erw[ (W_t -W) |\calf]_{\rr}$$ 
$$ \stackrel{t \to 0}{\longrightarrow} (-\Lambda W + R)_{\rr} = - \Lambda_{\rr} W_{\rr} + R_{\rr}.$$
For (ii), by \eqref{reellif} above, we have
\begin{eqnarray*}
&&\frac{1}{s(t)} \erw[ (W_t - W)_{\rr} (W_t - W)_{\rr}^T | \calf] \\
&=& \frac{1}{s(t)} \left( \frac12 \erw[ (W_t - W) (W_t - W)^{\ast} | \calf]_{\rr} + 
\frac12 \erw[ (W_t - W) (W_t - W)^{T} | \calf]_{\rr} J\right)\\
&\stackrel{t \to 0}{\longrightarrow}& \frac12 (2\Lambda \Sigma + S)_{\rr} + \frac12 T_{\rr} J = \Lambda_{\rr}
\Sigma_{\rr} + \frac12 (S_{\rr} + T_{\rr} J)\\
&=& 2 \Lambda_{\rr} \Sigma\tra + \frac12 (S_{\rr} + T_{\rr} J).
\end{eqnarray*}
Since $\| (W_t - W)_{\rr}\|_2 = \| W_t - W\|_2$, condition (iii) is trivially satisfied. So we may apply Proposition
\ref{Meckes-r} to the real version $W_{\rr}$ of $W$. Note that if $\tilde{Z}$ is distributed according to 
$\on{N}(0, \on{I}_{2d})$, then $(\Sigma\tra)^{1/2} \tilde{Z}$ is equal in distribution to $(\Sigma^{1/2} Z)_{\rr}$. Observing 
that $\| \Lambda^{-1}\|_{\on{op}} = \| \Lambda^{-1}_{\rr}\|_{\on{op}}, \|  S_{\rr}\|_{\on{HS}} = \sqrt{2} \|S\|_{\on{HS}},
\|T_{\rr} J\|_{\on{HS}} = \| T_{\rr}\|_{\on{HS}} = \sqrt{2} \|T\|_{\on{HS}},\ \erw\|R\|_2 = \erw\| R_{\rr}\|_2$ and
\newline $\|(\Sigma\tra)^{-1/2}\|_{\on{op}} = \sqrt{2} \| \Sigma_{\rr}^{-1}\|_{\on{op}} = \sqrt{2} \|\Sigma^{-1}\|_{\on{op}}$ yields the result.
\end{proof}

Now we construct a family of exchangeable pairs to fit into Meckes' set-up. We do this along the lines of Fulman's univariate
approach in \cite{Ful10}. 
In a nutshell, let $(M_t)_{t \ge 0}$ be Brownian motion on $K_n$ with
Haar measure as initial distribution. Set $M := M_0$. 
Brownian motion being reversible w.r.t.\ Haar measure, $(M, M_t)$ is an exchangeable pair for any $t > 0$. Suppose that the centered version $W$ of the statistic we are interested in is given by
$W = (f_1(M), \ldots, f_d(M))^T$ for suitable measurable functions $f_1, \ldots, f_d$. Defining 
$W_t = (f_1(M_t), \ldots, f_d(M_t))^T$ clearly yields an exchangeable pair $(W, W_t)$.  

To be more specific, let $\frk_n$ be the Lie algebra of $K_n$, endowed with the scalar 
product $\langle X, Y\rangle = \on{Tr}(X^*Y).$ Denote by $\Delta = \Delta_{K_n}$ the 
Laplace--Beltrami operator of $K_n$, i.e., the diffential operator corresponding to
the Casimir element of the enveloping algebra of $\frk_n$. Then $(M_t)_{t \ge 0}$ will
be the diffusion on $K_n$ with infinitesimal generator $\Delta$ and Haar measure as initial distribution. Reversibility then follows from general theory (see \cite[Sec.~II.2.4]{He84}, \cite[Sec.~V.4]{IkWa89}, for the relevant facts). Let $(T_t)_{t\ge 0}$, often symbolically 
written as $\left(e^{t \Delta}\right)_{t\ge 0}$, be the corresponding semigroup of 
transition operators on $\on{C}^2(K_n)$. From the Markov property of Brownian
motion we obtain that
\begin{equation}
 \label{bederwhg}
\erw[f(M_t)|M] = (T_tf)(M)\quad {\rm a.s.}
\end{equation}
Note that $(t, g) \mapsto (T_tf)(g)$ satisfies the heat equation, so a Taylor expansion in $t$
yields the following lemma, which is one of the cornerstones in Fulman's approach 
in that it shows that 
in first order in $t$ a crucial quantity for the regression conditions (i) of Propositions 
\ref{Meckes-r} and \ref{Meckes-c}
is given by the action of the Laplacian.  
\begin{lemma}
\label{entwicklung}
Let $f: K_n \to \cc$ be smooth. Then 
$$ \erw[f(M_t) | M] = f(M) + t (\Delta f)(M) + \on{O}(t^2).$$
\end{lemma}
\begin{remark}
 An elementary construction of the semigroup $(T_t)$ via an eigenfunction expansion in
irreducible characters of $K_n$ can be found in \cite{St70}. This construction immediately
implies Lemma \ref{entwicklung} for power sums, see Sec.\ \ref{Rains}, since they are
 characters of (reducible) tensor power representations.
\end{remark}
\begin{remark}
\label{convtype}
As $K_n$ is compact, Lemma \ref{entwicklung} implies that $\lim_{t\to 0} \erw[f(M_t) - f(M)| M] = 0$ a.s.\
and in $L^2(\pp)$. Arguments of this type will occur frequently in what follows, usually without further notice.
\end{remark}

\section{Power sums and Laplacians}
\label{Rains}

For indeterminates $X_1, \ldots, X_n$ and a finite family
$\lambda = (\lambda_1, \ldots, \lambda_r)$ of positive integers, the
power sum symmetric polynomial with index $\lambda$ is given by $$p_{\lambda} := \prod_{k=1}^r \sum_{j=1}^n X_j^{\lambda_k}.$$
For $A \in \cc^{n \times n}$ with eigenvalues $c_1, \ldots, c_n$ (not necessarily distinct), we write $p_{\lambda}(A)$ in the place of $p_{\lambda}(c_1, \ldots, c_n)$. Then one has the identity
$$ p_{\lambda}(A) = \prod_{k=1}^r \on{Tr}(A^{\lambda_k}).$$ 
Using this formula, we will extend the definition of $p_{\lambda}$ to integer indices by $p_0(A) = \on{Tr}(I)$, $p_{-k}(A)
= \on{Tr}((A^{-1})^k)$ and $p_{\lambda}(A) = \prod_{j=1}^r p_j(A).$

In particular, the $p_{\lambda}$ may be viewed as functions
on $K_n$, and the action of the Laplacian $\Delta_{K_n}$ on them, useful in view of Lemma \ref{entwicklung}, is available from \cite{Rai97, Lev08}. We specialize their formulae to the cases that will be needed in what follows:
\begin{lemma}
\label{rains-u}
For the Laplacian $\Delta_{\on{U}_n}$ on $\on{U}_n$, one has
\begin{eqnarray*}
(i) \quad \Delta_{\on{U}_n} p_j &=& -n j p_j - j \sum_{l = 1}^{j-1} p_{l, j-l}.\\
(ii)\quad \Delta_{\on{U}_n} p_{j, k} &=& - n (j+k) p_{j, k} - 2 jk p_{j+k} - j p_k \sum_{l=1}^{j-1} p_{l, j-l}\nonumber\\ 
&-& k p_j \sum_{l=1}^{k-1} p_{l, k-l}.\\
(iii) \quad \Delta_{\on{U}_n}(p_j \overline{p_k}) &=& 2jk p_{j-k} - n(j+k) p_j \overline{p_k} - j \overline{p_k}
\sum_{l=1}^{j-1} p_{l, j-l}\nonumber\\ &-& k p_j  \sum_{l=1}^{k-1} \overline{p_{l, k-l}}. 
\end{eqnarray*}
\end{lemma}

\begin{lemma}
\label{Rains-so}
For the Laplacian $\Delta_{\on{SO}_n}$ on $\on{SO}_n$,
\begin{eqnarray*}
(i)\hfill \quad \Delta_{\on{SO}_n} p_{j} &=& - \frac{(n-1)}{2} j p_j - \frac{j}{2} \sum_{l=1}^{j-1} p_{l, j-l} +
\frac{j}{2} \sum_{l=1}^{j-1} p_{2l - j}.\\
(ii)\hfill \quad \Delta_{\on{SO}_n} p_{j,k} &=& - \frac{(n-1)(j+k)}{2}\ p_{j,k}  - \frac{j}{2}\ p_k \sum_{l=1}^{j-1} p_{l, j-l} 
 \hfill - \frac{k}{2}\ p_j\ \sum_{l=1}^{k-1} p_{l, k-l}\ -\ jk\ p_{j+k}
\\ &&+\ \frac{j}{2}\ p_k\ \sum_{l=1}^{j-1} p_{j-2l}\ +\ \frac{k}{2}\ p_j\ \sum_{l=1}^{k-1} p_{k - 2l}\ +\ jk\ p_{j-k}.
\end{eqnarray*}
\end{lemma}

\begin{lemma}
\label{Rains-sp}
For the Laplacian $\Delta_{\on{USp}_{2n}}$ on $\on{USp}_{2n}$,
\begin{eqnarray*}
(i)\hfill \quad \Delta_{\on{USp}_{2n}} p_{j} &=& - \frac{(2n+1)}{2}\ j p_j - \frac{j}{2} \sum_{l=1}^{j-1} p_{l, j-l} -
\frac{j}{2} \sum_{l=1}^{j-1} p_{2l - j}.\\
(ii)\hfill \quad  \Delta_{\on{USp}_{2n}} p_{j,k} &=& - \frac {(2n + 1)}{2}\ (j+k)\ p_{j,k} 
-\ jk \ p_{j+k} 
 - \frac{j}{2}\ p_k \sum_{l=1}^{j-1} p_{l, j-l} 
  - \frac{k}{2}\ p_j\ \sum_{l=1}^{k-1} p_{l, k-l}\ 
\\ &&-\ \frac{j}{2}\ p_k\ \sum_{l=1}^{j-1} p_{j-2l}\ - \ \frac{k}{2}\ p_j\ \sum_{l=1}^{k-1} p_{k - 2l}\ +\ jk\ p_{j-k}.
\end{eqnarray*}

\end{lemma}

In what follows, we will need to integrate certain $p_\lambda$ over the group $K_n$. Thus we will need special cases of the 
Diaconis-Shahshahani moment formulae that we now recall (see \cite{DiaSh94, HuRud03, PaVa04, Sto05} for proofs). Let $a = (a_1, \ldots, a_r)$, $b = (b_1, \ldots, b_q)$ be families of nonnegative integers and define

$$ f_a(j) := \left\{ \begin{array}{ll} 1 & {\rm if}\ a_j = 0,\\ 
0 & {\rm if}\ ja_j\ {\rm is~odd}, a_j \ge 1,\\
j^{a_j/2} (a_j - 1)!! & {\rm if}\ j\ {\rm is~odd~and}\ a_j\ {\rm is~even}, a_j \ge 2,\\
1 + \sum_{d=1}^{\lfloor a_j/2 \rfloor} j^d { a_j \choose 2d} (2d - 1)!!&
{\rm if}\ j\ {\rm is~even}, a_j \ge 1.\end{array}\right.$$
Here we have used the notation $(2m- 1)!! = (2m-1)(2m-3) \cdot \ldots \cdot 3 \cdot 1.$ Further, we will write
$$ k_a := \sum_{j=1}^r j a_j,\ k_b := \sum_{j=1}^q j b_j,\quad {\rm and}\quad \eta_j := \left\{ \begin{array}{ll} 1, & {\rm if}\ j\ {\rm is~even},\\
0,& {\rm if}\ j\ {\rm is~odd}.\end{array}\right.$$

\begin{lemma}
\label{DS-u} Let $M = M_n$ be a Haar-distributed element of $\on{U}_n$, $Z_1, \ldots, Z_r$ iid complex standard
normals. Then, if $k_a \neq k_b$, 
$$ \alpha_{(a, b)} := \erw \left( \prod_{j=1}^r (\on{Tr}(M^j))^{a_j} \prod_{j=1}^q \overline{(\on{Tr}(M^j))}^{b_j} \right)
=0.$$ If $k_a = k_b$ and $n \ge k_a$, then
$$\alpha_{(a, b)} = \delta_{a, b} \prod_{j=1}^r j^{a_j} a_j! = \erw\left( \prod_{j=1}^r (\sqrt{j} Z_j)^{a_j}\  
\prod_{j=1}^q \overline{(\sqrt{j} Z_j)}^{b_j} \right).$$

\end{lemma}

\begin{lemma}
\label{DS-so}
If $M = M_n$ is a Haar-distributed element of $\on{SO}_n$, $n- 1 \ge k_a$, $Z_1, \ldots, Z_r$ iid real standard
normals, then
\begin{equation}
\label{dsform-so}
\erw \left( \prod_{j=1}^r (\on{Tr}(M^j))^{a_j}\right) = \erw \left( \prod_{j=1}^r (\sqrt{j} Z_j + \eta_j)^{a_j}\right)  
= 
\prod_{j=1}^r f_a(j).
\end{equation}
\end{lemma}

\begin{lemma}
\label{DS-sp}
If $M = M_{2n}$ is a Haar-distributed element of $\on{USp}_{2n}$, $2n \ge k_a$, $Z_1, \ldots, Z_r$ iid real standard
normals, then
\begin{equation}
\label{dsform-sp}
\erw \left( \prod_{j=1}^r (\on{Tr}(M^j))^{a_j}\right) = \erw \left( \prod_{j=1}^r (\sqrt{j} Z_j - \eta_j)^{a_j}\right)  
= 
\prod_{j=1}^r (-1)^{(j-1)a_j}f_a(j).
\end{equation}
\end{lemma}

\section{The unitary group}
\label{sec-u}

Let $M = M_n$ be distributed according to Haar measure on $K_n = \on{U}_n$.  For $d \in \nn,\ r = 1, \ldots, d,$ consider
the $r$-dimensional complex random vector 
$$W := W(d, r, n) := \left(\on{Tr}(M^{d-r+1}), \on{Tr}(M^{d-r+2}), \ldots, \on{Tr}(M^d)\right)^T.$$ Let $Z := (Z_{d-r+1}, \ldots, Z_d)^T$ be an $r$-dimensional standard complex normal random vector, i.e., there are iid 
real random variables $X_{d-r+1}, \ldots, X_d,$ $Y_{d-r+1}, \ldots, Y_d$ with distribution $\on{N}(0, 1/2)$ such that $Z_j = X_j + i Y_j$ for $j = d-r+1, \ldots, d$. Furthermore, we take $\Sigma$ to denote the diagonal matrix $\on{diag}(d-r+1, d-r+2, \ldots, d)$ and write 
$Z_{\Sigma} := \Sigma^{1/2} Z.$ The present section is devoted to the proof of the following

\begin{theorem}
\label{thm-u} If $n \ge 2d$, the Wasserstein distance between $W$ and $Z_{\Sigma}$ is
\begin{equation} 
\label{ordnungsformel}
d_\calw(W,Z_\Sigma)=\on{O}\left(\frac{\max\left\{\frac{r^{7/2}}{(d-r+1)^{3/2}},\,(d-r)^{3/2}\sqrt{r}\right\}}{n}\right)\,.\end{equation}
In particular, for $r = d$ we have
$$  d_{\calw}(W, Z_{\Sigma}) = \on{O}\left(\frac{d^{7/2}}{n}\right),$$
and for $r \equiv 1$
$$  d_{\calw}(W, Z_{\Sigma}) = \on{O}\left(\frac{d^{3/2}}{n}\right).$$
If $1 \le r = \lfloor cd\rfloor$ for $0 < c < 1$, then 
$$  d_{\calw}(W, Z_{\Sigma}) = \on{O}\left(\frac{d^{2}}{n}\right).$$
\end{theorem} 

\begin{remark}
\label{fulmanvergleich}
The case $r \equiv 1$ means that one considers a single power $\on{Tr}(M^d)$. In his study \cite{Ful10}
of the univariate case, Fulman considers the random variable  $\frac{\on{Tr}(M^d)}{\sqrt{d}}$ instead. 
By the scaling properties of the Wasserstein metric, the present result implies 
$$ d_{\calw}\left( \frac{\on{Tr}(M^d)}{\sqrt{d}}, \on{N}(0, 1)\right) = \on{O}\left(\frac{d}{n}\right).$$ So in this 
special case we recover
the rate of convergence that was obtained by Fulman, albeit in Wasserstein rather than Kolmogorov distance.

\end{remark}

In order to prove Theorem \ref{thm-u}, we invoke the construction that was explained in Section \ref{Meckes} above, 
yielding a family $(W, W_t)_{t > 0}$ 
of exchangeable pairs such that $W_t = W_t(d, r, n) = 
(\on{Tr}(M_t^{d-r+1}), \on{Tr}(M_t^{d-r+2}), \ldots, \on{Tr}(M_t^d))^T$. We have to check the conditions of Prop.\ \ref{Meckes-c}. As to (i), from Lemma \ref{entwicklung} and Lemma \ref{rains-u}(i) we obtain
\begin{eqnarray*}
\erw[p_j(M_t) | M] &=& p_j(M) + t (\Delta p_j)(M) + \on{O}(t^2) \\ 
&=& p_j(M) - t n j p_j(M) - t j \sum_{l=1}^{j-1}p_{l, j-l}(M) + \on{O}(t^2),
\end{eqnarray*}
hence $$ \erw[ p_j(M_t) - p_j(M) | M] = -t \left( njp_j(M) + j \sum_{l=1}^{j-1}
p_{l, j-l}(M) + \on{O}(t)\right).$$ From this we see that
$$\frac{1}{t} \erw[W_t - W|M] = \frac{1}{t} \left( \begin{array}{l}
\erw[p_{d-r+1}(M_t)- p_{d-r+1}(M) | M]\\
\erw[p_{d-r+2}(M_t)- p_{d-r+2}(M) | M]\\
\vdots\\
\erw[p_d(M_t)- p_d(M) | M]\end{array} \right)$$ 
tends, as $t \to 0$, a.s.\ and in $\on{L}^1(\|\cdot\|_{\rm HS})$ to
$$  - \left( \begin{array}{l} (d-r+1)\left[ n p_{d-r+1} + \sum_{l=1}^{d-r}
p_{l, d-r+1-l}\right](M)
 \\ (d-r+2) \left[n p_{d-r+2} + \sum_{l=1}^{d-r+1}
p_{l, d-r+2-l}\right](M) \\ \vdots \\ d \left[n p_d + 
\sum_{l=1}^{d-1} p_{l, d-l}\right](M) 
         \end{array}\right) =:  - \Lambda W + R,$$ where $\Lambda = 
\on{diag}(nj:\ j = d-r+1, \ldots, d)$ and $R = (R_{d-r+1}, \ldots, R_d)^T$ with $R_j = -j \sum_{l=1}^{j-1} p_{l, j-l}(M).$
(See Remark \ref{convtype} above for the type of convergence.)\\

To verify conditions (ii) and (iii) of Prop.\ \ref{Meckes-c}, we first prove the following lemma.

\begin{lemma} For $j, k = d-r+1, \ldots, d$, one has
\label{zuwkorrel-u}
\begin{itemize}
 \item[(i)]
$$\erw[ (p_j(M_t) - p_j(M))(p_k(M_t) - p_k(M)\ |\ M] =  
- 2 tjk\ p_{j+k}(M) + \on{O}(t^2)\\
$$ 
\item[(ii)]
\begin{eqnarray*}
&&\erw[(p_j(M_t) - p_j(M)) (\overline{p_k}(M_t) - \overline{p_k}(M))\ |\ M]
= 2tjk\ p_{j-k}(M) + \on{O}(t^2)
\end{eqnarray*}
\end{itemize}
\end{lemma}

\begin{proof}
By well known properties of conditional expectation, 
\begin{eqnarray*}
 &&\erw[ (p_j(M_t) - p_j(M)) (p_k(M_t) - p_k(M)) | M]\\
&=& \erw[p_j(M_t) p_k(M_t) | M] - p_j(M) \erw[p_k(M_t)| M] - p_k(M) \erw[p_j(M_t)| M] +
p_j(M) p_k(M).
\end{eqnarray*}
Applying Lemma \ref{entwicklung} and \ref{rains-u}(ii) to the first term yields
$$ \erw[ p_{j, k}(M_t)| M] = p_{j, k}(M) + t (\Delta p_{j, k})(M) + \on{O}(t^2)$$
$$ = p_{j, k}(M) + t\left(-n(j+k) p_{j,k}(M) - 2 jk p_{j+k}(M) - j p_k(M) 
\sum_{l=1}^{j-1} p_{l, j-l}(M)\right.$$ $$\left. - k p_j(M) \sum_{l=1}^{k-1} p_{l, k-l}(M)\right)
+ \on{O}(t^2).$$
Analogously, for the second term, 
\begin{eqnarray*}
&& p_j(M) \erw[p_k(M_t)|M] = p_j(M) (p_k(M) + t(\Delta p_k)(M) + \on{O}(t^2))\\
&=& p_{j, k}(M) + t p_j(M) \left( - nk p_k(M) - k \sum_{l=1}^{k-1} p_{l, k-l}(M)\right) + \on{O}(t^2)\\
&=& p_{j, k}(M) - tnk p_{j, k}(M) - tk p_j(M) \sum_{l=1}^{k-1} p_{l, k-l}(M) + \on{O}(t^2),
\end{eqnarray*}
and by symmetry
$$ p_k(M) \erw[p_j(M_t)|M] = p_{j, k}(M) - tnj p_{j, k}(M) - tj p_k(M) \sum_{l=1}^{j-1} p_{l, j-l}(M) + \on{O}(t^2).$$
Summing up, one obtains
$$ \erw[ (p_j(M_t) - p_j(M))(p_k(M_t) - p_k(M)) | M] = 
- 2 tjk p_{j+k}(M) + \on{O}(t^2),$$ proving the first assertion. For the second, we compute 
analogously
\begin{eqnarray*}
 &&\erw[(p_j(M_t) - p_j(M)) (\overline{p_k}(M_t) - \overline{p_k}(M)| M]\\
&=& \erw[p_j \overline{p_k}(M_t) | M] - p_j(M) \erw[\overline{p_k}(M_t)|M]
- \overline{p_k}(M) \erw[p_j(M_t)|M] + p_j \overline{p_k}(M),
\end{eqnarray*}
and we have by Lemma \ref{rains-u}(iii)
\begin{eqnarray*}
 &&\erw[p_j \overline{p_k}(M_t)| M] = p_j \overline{p_k}(M) + t \Delta(p_j \overline{p_k})(M)
+ \on{O}(t^2) \\ &=& p_j \overline{p_k}(M) + t\left( 2jk p_{j-k} - n(j+k) p_j 
\overline{p_k}(M)\right.\\ && \left. - j \overline{p_k}(M) \sum_{l=1}^{j-1} p_{l, j-l}(M) - k p_j(M)
\sum_{l=1}^{k-1} \overline{p_{l, k-l}}(M)\right) + \on{O}(t^2)
\end{eqnarray*}
as well as
\begin{eqnarray*}
&& p_j(M) \erw[\overline{p_k}(M_t)|M] = p_j(M) \left( \overline{p_k}(M) + t (\Delta \overline{p_k})(M) + \on{O}(t^2)\right)\\ &=& p_j \overline{p_k}(M) + t p_j(M) \left(
- n k \overline{p_k}(M) - k \sum_{l=1}^{k-1} \overline{p_{l, k-l}}(M) \right) + \on{O}(t^2)
\end{eqnarray*}
and
\begin{eqnarray*}
 &&\overline{p_k}(M) \erw[p_j(M_t)|M] = \overline{p_k}(M) \left( p_j(M) + t(\Delta p_j)(M) +
\on{O}(t^2)\right) \\ &=& p_j \overline{p_k}(M) + t~ \overline{p_k}(M) \left( - nj p_j(M) -
j \sum_{l=1}^{j-1} p_{l, j-l}(M)\right) + \on{O}(t^2).
\end{eqnarray*}
Summing up, one has
$$ \erw[ ( p_j(M_t) - p_j(M))( \overline{p_k}(M_t) - \overline{p_k}(M))| M] = 2tjk\
p_{j-k} + \on{O}(t^2).$$
\end{proof}

Now we are in a position to identify the random matrices $S, T$ of Prop.\ \ref{Meckes-c}. By
Lemma \ref{zuwkorrel-u}, $\frac{1}{t} \erw[(W_t - W)(W_t - W)^T| M]$ converges almost surely, and in $\on{L}^1(\| \cdot\|_{\rm HS})$, to $T = (t_{jk})_{ j, k = d-r+1, \ldots, d}$, where $t_{jk} = - 2jk p_{j+k}(M) $ for $j, k = d-r+1, \ldots, d$. Observing that
$\Lambda \Sigma = \on{diag}(n j^2:\ j = d-r+1, \ldots, d)$, one has that 
$\frac{1}{t} \erw[(W_t - W)(W_t - W)^* | M]$ converges almost surely, and in $\on{L}^1(\| \cdot\|_{\rm HS})$, to $2 \Lambda \Sigma + S$, where $S = (s_{jk})_{j, k = d-r+1, \ldots, d}$ is given by

$$ s_{jk} = \left\{ \begin{array}{ll} 0, & \text{\rm if}\ j= k,\\
                     2jk\ p_{j-k}(M), & \text{\rm if}\ j \neq k.
                    \end{array}\right. 
$$
Next we will verify condition (iv), using Remark \ref{rem-meckesc}. 
Specifically, we will show that $\erw[ \| W_t - W\|^3_2] = \on{O}(t^{3/2}).$
Since 
\begin{eqnarray*}
 \erw[ \| W_t - W\|^3_2] &\le& \sum_{j, k, l = d-l+1}^{d} \erw[ | (W_{t, j} - W_j)
(W_{t, k} - W_k)(W_{t, l} - W_l)|]\\
&\le& \sum_{j, k, l = d-r+1}^{d} \left( \erw[| W_{t,j} - W_j|^3]
\erw[| W_{t,k} - W_k|^3]\erw[| W_{t,l} - W_l|^3]\right)^{1/3},
\end{eqnarray*}
it suffices to prove that $\erw[| W_{t,j} - W_j|^3] = \on{O}(t^{3/2})$ for all $j = d-r+1, \ldots, d$. This in turn follows from the next lemma, since 
$$ \erw[| W_{t,j} - W_j|^3] \le \left(\erw[| W_{t,j} - W_j|^4]\erw[| W_{t,j} - W_j|^2]\right)^{1/2}.$$

\begin{lemma}
 \label{zuw-u-2-4} For $j = d-r+1, \ldots, d$, $n \ge 2d$,
\begin{itemize}
 \item[(i)] $ \erw[| W_{t,j} - W_j|^2] = 2 j^2 nt + \on{O}(t^2)$,
\item[(ii)] $\erw[| W_{t,j} - W_j|^4] = \on{O}(t^2).$
\end{itemize}
\end{lemma}
\begin{proof}
 By Lemma \ref{zuwkorrel-u} (ii),
$$ \erw[| W_{t,j} - W_j|^2] = \erw[ (p_j(M_t) - p_j(M)) (\overline{p_j}(M_t) - \overline{p_j}(M))] = \erw[ 2t j^2 n + t^2)] = 2 t j^2 n + \on{O}(t^2),$$
establishing (i). Turning to (ii), one calculates that
\begin{eqnarray*}
 && \erw[| W_{t,j} - W_j|^4]\\ &=& \erw\left[ (W_{t, j} - W_j)^2 \cdot
(\overline{W_{t, j}} - \overline{W_j})^2\right]\\ &=& \erw\left[ (p_j(M_t) - p_j(M))^2\ 
(\overline{p_j}(M_t) - \overline{p_j}(M))^2\right]\\
&=& \erw\left[ (p_j^2(M_t) - 2p_j(M) p_j(M_t) + p_j^2(M))
(\overline{p_j}^2(M_t) - 2\overline{p_j}(M) \overline{p_j}(M_t) + \overline{p_j}^2(M))\right]\\
&=& \erw\left[ p_j^2(M_t) \overline{p_j}^2(M_t)\right] - 2 \erw\left[ p_j^2(M_t) 
\overline{p_j}(M_t)\overline{p_j}(M)\right]
- 2 \erw\left[ \overline{p_j}^2(M_t) 
p_j(M_t)p_j(M)\right]\\
&& + \erw\left[ p_j^2(M_t) \overline{p_j}^2(M)\right] 
+ \erw\left[ \overline{p_j}^2(M_t) p_j^2(M)\right] + 4 \erw\left[ p_j(M_t) \overline{p_j}(M_t)
p_j(M) \overline{p_j}(M)\right]\\
&&- 2 \erw\left[ p_j^2(M) \overline{p_j}(M) \overline{p_j}(M_t)\right]
 - 2 \erw\left[ \overline{p_j}^2(M) p_j(M) p_j(M_t)\right] + \erw\left[ p_j^2(M) \overline{p_j}^2(M)\right]\\
&=:& S_1 - 2 S_2 - 2 S_3 + S_4 + S_5 + 4S_6 - 2 S_7 - 2 S_8 + S_9.
\end{eqnarray*}
By exchangeability, we have $S_1 = S_9$, $S_3 = \overline{S_2}$, $S_4 = S_5$, 
$S_7 = S_2$, $S_8 = \overline{S_7} = S_3$.

Now, for $n \ge 2d$, i.e., large enough for the moment formulae of Lemma \ref{DS-u} to apply for all $j = d-r+1, \ldots, d$, 
\begin{eqnarray*}
S_1 &=& 2j^2,\\
 S_8 &=& \erw\left[ \overline{p_j}^2(M) p_j(M) p_j(M_t)\right] = 
\erw\left[ \overline{p_j}^2(M) p_j(M) \erw\left[p_j(M_t)| M\right]\right]\\
&=& \erw\left[ \overline{p_j}^2(M) p_j(M) \left(p_j(M) + t (\Delta p_j)(M) + \on{O}(t^2)\right)\right]\\
&=& \erw\left[ \overline{p_j}^2(M) p_j^2(M)\right] + \erw\left[ \overline{p_j}^2(M)
p_j(M) \on{O}(t^2) \right]\\  && +\ t~\erw\left[ \overline{p_j}^2(M) p_j(M) \left( - n j p_j(M) 
-j \sum_{l=1}^{j-1} p_{l, j-l}(M)\right) \right]\\
&=& 2 j^2 + \on{O}(t^2) - tnj 2j^2 - tj \sum_{l=1}^{j-1} \erw \left[ \overline{p_j}^2(M) p_j(M) 
p_{l, j-l}(M)\right] = 2j^2 - 2t n j^3 + \on{O}(t^2).
\end{eqnarray*}
Hence, $S_3 = S_2 = \overline{S_7} = S_8 = 2j^2 - 2t n j^3 + \on{O}(t^2).$ On the other hand, 
\begin{eqnarray*}
S_4 &=& \erw\left[ p_j^2(M_t) \overline{p_j}^2(M)\right] =
\erw\left[ \overline{p_j}^2(M) \erw\left[ p_j^2(M_t) | M\right] \right]\\
&=& \erw\left[ \overline{p_j}^2(M) \left( p_j^2(M) + t~(\Delta p_j^2)(M) + \on{O}(t^2)\right) \right] 
= \erw\left[ \overline{p_j}^2(M) p_j^2(M)\right] + \on{O}(t^2)\\ 
&& +\ t~\erw\left[ \overline{p_j}^2(M) \left( - 2njp_j^2(M) - 2 j^2 p_{2j}(M) - 2j p_j(M) 
\sum_{l=1}^{j-1} p_{l, j-l}(M)\right)\right]\\
&=& 2j^2 - 4ntj^3 - 2 t j^2 \erw\left[\overline{p_j}^2(M) p_{2j}(M)\right]
- 2tj \sum_{l=1}^{j-1} \erw\left[\overline{p_j}^2(M) p_j(M) p_{l, j-l}(M)\right] + 
\on{O}(t^2)\\
&=& 2j^2 - 4tnj^3 + \on{O}(t^2).
\end{eqnarray*}
Finally,
\begin{eqnarray*}
 S_6 &=& \erw\left[ p_j(M_t) \overline{p_j}(M_t) p_j(M) \overline{p_j}(M)\right]
= \erw\left[ p_j(M) \overline{p_j}(M) \erw\left[ p_j(M_t) \overline{p_j}(M_t)\ |\ M\right] \right]\\
&=& \erw\left[ p_j(M) \overline{p_j}(M) \left( p_j(M) \overline{p_j}(M) +
t~ (\Delta p_j \overline{p_j})(M) + \on{O}(t^2)\right) \right]\\
&=& \erw\left[ p_j^2(M) \overline{p_j}^2(M) \right] + \on{O}(t^2)
+ t~ \erw\Bigg[ p_j(M) \overline{p_j}(M)\\ &&  \left( 2j^2n - 2nj p_j(M) \overline{p_j}(M)
- j p_j(M) \sum_{l=1}^{j-1} \overline{p_{l, j-l}}(M) - j \overline{p_j}(M) 
\sum_{l=1}^{j-1} p_{l, j-l}(M) \right) \Bigg]\\
&=& 2j^2 + 2tj^2n \erw\left[ p_j(M) \overline{p_j}(M)\right] - 2ntj \erw\left[
p_j^2(M) \overline{p_j}^2(M)\right]\\ && - tj \sum_{l=1}^{j-1} \erw\left[ p_j^2(M) \overline{p_j}(M) 
\overline{p_{l, j-l}}(M)\right]  - tj  \sum_{l=1}^{j-1}  \erw\left[ p_j(M) p_{l, j-l}(M)
\overline{p_j}^2(M)\right] + \on{O}(t^2) \\
&=& 2 j^2 + 2tnj^3 - 4ntj^3 + \on{O}(t^2) = 2 j^2 - 2nt j^3 + \on{O}(t^2).
\end{eqnarray*}
Putting the pieces together, 
\begin{eqnarray*}
 \erw\left[ |W_{t, j}- W_j|^4\right] &=& 2\cdot j^2 - 8(2j^2 - 2tnj^3)
+ 2(2j^2 - 4tnj^3) + 4(2j^2 - 2ntj^3) + \on{O}(t^2)\\
&=&j^2(4-16+4+8) + tnj^3(16-8-8) +  \on{O}(t^2) = \on{O}(t^2).
\end{eqnarray*}
as asserted.
\end{proof}

With the conditions of Theorem \ref{Meckes-c} in place, we have
\begin{equation}
\label{plugin-u}
d_{\calw}(W, \Sigma^{1/2} Z) \le \| \Lambda^{-1}\|_{\rm op}
\left( \erw\| R \|_2  + \frac{1}{\sqrt{2 \pi}} \| \Sigma^{-1/2}\|_{\rm op}
\erw\left[ \| S\|_{\rm HS} + \| T\|_{\rm HS}\right] \right).
\end{equation} 
To bound the quantities on the right hand side, we first observe that 
$\| \Lambda^{-1}\|_{\rm op} = \frac{1}{n (d-r+1)}$ and $ \| \Sigma^{-1/2}\|_{\rm op} = \frac{1}{\sqrt{d-r+1}}$. 
Now $$ \erw \| R\|_2^2 = \sum_{j=d-r+1}^d \erw R_j \overline{R_j} = \sum_{j=d-r+1}^d j^2 
\sum_{l, m = 1}^{j-1} \erw\left[ p_{l, j-l}(M) \overline{p_{m, j-m}}(M)\right].$$
For $n \ge d$, Lemma \ref{DS-u} implies
$$ \erw\left[ p_{l, j-l}(M) \overline{p_{m, j-m}}(M) \right] = \delta_{l, m} 
\erw\left[ p_{l, j-l}(M) \overline{p_{l, j-l}}(M)\right]$$
and $$ \erw\left[ p_{l, j-l}(M) \overline{p_{l, j-l}}(M)\right] = 
\left\{ \begin{array}{ll} 2l(j-l), & l = \frac{j}{2}\\ l(j-l), & {\rm otherwise} 
        \end{array} \right\}
\le 2l (j-l).$$ Hence
\begin{eqnarray*}
 \erw \| R \|_2^2 &\le& 2 \sum_{j=d-r+1}^d j^2 \sum_{l=1}^{j-1} l(j-l) \\
&=& 2 \sum_{j=d-r+1}^d j^2 \frac{j\ (j^2-1)}{6} \\
&=& \frac13 \sum_{j=d-r+1}^d (j^5 - j^3)\\
&=& \frac13 \sum_{k=1}^r \left( (k+d-r)^5 - (k+d-r)^3\right)\\
&=& \frac13 \sum_{k=1}^r (k^5 + 5k^4(d-r) + 10 k^3(d-r)^2 + 10 k^2 (d-r)^3 + 5 k(d-r)^4 + (d-r)^5)\\
&-& \frac13 \sum_{k=1}^r (k^3 + 3 k^2 (d-r) + 3k(d-r)^2 + (d-r)^3) \\
&=& \frac13 \left[ \sum_{k=1}^r k^5 + (d-r) \left( 5 \sum_{k=1}^r k^4 - 3 \sum_{k=1}^r k^2\right)\right.\\
&+& (d-r)^2 \left( 10 \sum_{k=1}^r k^3 - 3 \sum_{k=1}^r k\right) + (d-r)^3 \left( 10 \sum_{k=1}^r k^2 - r \right)\\
&+& \left.(d-r)^4\ 5 \sum_{k=1}^r k + r (d-r)^5 \right]\\
&\le & \frac13 \left[ r^6 + 5 (d-r) r^5 + 10 (d-r)^2 r^4 + 10 (d-r)^3 r^3 + 5 (d-r)^4 r^2 + (d-r)^5 r\right]\\
&=& \on{O}(\max(r^6, r(d-r)^5)),
\end{eqnarray*}
hence $\erw\|R\|_2 \le \sqrt{ \erw\|R\|_2^2} = \on{O}(\max(r^3, \sqrt{r} (d-r)^{5/2})).$
On the other hand, 
\begin{eqnarray*}
 \erw \| S \|_{\rm HS}^2 &=& \sum_{j, k = d-r+1}^d \erw |s_{jk}|^2\\
&=& 8 \sum_{d-r+1 \le k < j\le d} j^2 k^2 \erw\left[ p_{j-k}(M) \overline{p_{j-k}}(M)\right]
= 8 \sum_{d-r+1 \le k < j\le d} j^2 k^2 (j-k)
\end{eqnarray*}
\begin{eqnarray*}
&=& 8 \sum_{d-r+1 \le k < j \le d} j^3 k^2 - 8 \sum_{d-r+1 \le k < j \le d} j^2 k^3\\
&=& 8 \sum_{d-r+2}^d j^3 \frac{(j-1) j (2j-1)}{6} - j^2 \frac{(j-1)^2 j^2}{4} = \frac23 \sum_{d-r+2}^d j^6 - j^4\\
&\le & \frac23 \sum_{j=2}^r (j + d-r)^6 \\
&=& \frac23 \sum_{j=2}^r \left[ j^6 + 6 j^5 (d-r) + 15 j^4 (d-r)^2 + 20 j^3 (d-r)^3 + 15 j^2 (d-r)^4 
+ 6 j (d-r)^5 + (d-r)^6 \right]\\
&\le & r^7 + 6 r^6 (d-r) + 15 r^5 (d-r)^2 + 20 r^4 (d-r)^3 + 15 r^3 (d-r)^4 + 6 r^2 (d-r)^5 + r (d-r)^6\\
&=& \on{O}(\max(r^7, r (d-r)^6)).
\end{eqnarray*}
So we have obtained that
$$ \erw \| S \|_{\rm HS} \le \sqrt{\erw \| S \|_{\rm HS}^2} = \on{O}(\max(r^{7/2}, \sqrt{r} (d-r)^3)).$$

As to $\| T\|_{\rm HS}$, for $n \ge 2d$ we have
\begin{eqnarray*}
\erw[ \| T\|_{\rm HS}^2] &=& \sum_{j, k = d-r+1}^d \erw[ |t_{ik}|^2] = \sum_{j, k = d-r+1}^d 4 j^2 k^2 \erw[ p_{j+k}(M) \overline{p_{j + k}}(M)]\\
&=& \sum_{j, k = d-r+1}^d 4 j^2 k^2(j+k) = 4 \sum_{j, k = d-r+1}^d j^3 k^2 + 
4 \sum_{j, k = d-r+1}^d k^3 j^2\\ &=& 8 \left( \sum_{j=d-r+1}^d j^3\right)
\left( \sum_{j=d-r+1}^d k^2\right) \\
&=& 8 \left( \sum_{j=1}^r (j+d-r)^3\right) \left( \sum_{k=1}^r (k+d-r)^2 \right)\\
&=& 8 \left(\sum_{j=1}^r j^3 + 3 j^2(d-r) + 3 j (d-r)^2 + (d-r)^3 \right) \left( \sum_{k=1}^r k^2 + 2k(d-r) + (d-r)^2\right)\\
&\le & 8(r^4 + 3 r^3 (d-r) + 3(r^2 (d-r)^2 + r (d-r)^3)( r^3 + 2 r^2 (d-r) + r(d-r)^2)\\
&= & \on{O}(\max(r^7, r^2 (d-r)^5)). 
\end{eqnarray*}
hence
$$\erw \| T\|_{\rm HS} = \on{O}(\max(r^{7/2}, r (d-r)^{5/2}).$$
Plugging these bounds into \eqref{plugin-u}, we obtain
\begin{eqnarray}
&&d_{\calw}(W, \Sigma^{1/2}Z) \nonumber \\ &\le & \frac{1}{n(d-r+1)} \left( \on{O}\left( \max(r^3, \sqrt{r} (d-r)^{5/2}\right)
+ \frac{ \on{O}\left( \max(r^{7/2}, \sqrt{r}(d-r)^3) + \max(r^{7/2}, r(d-r)^{5/2}\right)}{\sqrt{d-r+1}} \right)
\nonumber\\
&=&  \on{O}\left( \frac{\sqrt{d-r+1} \max(r^3, \sqrt{r} (d-r)^{5/2}) +  \max(r^{7/2}, \sqrt{r}(d-r)^3) + \max(r^{7/2}, r(d-r)^{5/2})}{n(d-r+1)^{3/2}} \right) \label{un-summingup}.
\end{eqnarray}
In the case that $r \le d-r$, 
\begin{eqnarray*}
\text{\rm \eqref{un-summingup}} &=& \on{O}\left( \frac{ \sqrt{ r(d-r+1)} (d-r)^{5/2}}{ n (d-r+1)^{3/2}} \right) =
\on{O}\left( \frac{\sqrt{r} (d-r)^{3/2}}{n}\right),
\end{eqnarray*}
whereas in the case that $d-r < r$ one obtains
$$ \text{\rm \eqref{un-summingup}} = \on{O}\left( \frac{r^{7/2}}{ n (d-r+1)^{3/2}} \right).$$
This yields the first claim, and the others follow easily.

\section{The special orthogonal group}
\label{sec-so}
Let $M = M_n$ be distributed according to Haar measure on $K_n = \on{SO}_n$. For $d \in \nn,\ r = 1, \ldots, d,$ consider the random vector
$W:=W(d,r,n):=(f_{d-r+1}(M),f_{d-r+2}(M),\ldots,f_d(M))^T,$ where 
\[f_j: \on{SO}_n\rightarrow\rr\,,\quad f_j=
\begin{cases}
p_j,&  j\text{ odd}\\
p_j-1,& j\text{ even.}
\end{cases}\]

Let $Z=(Z_{d-r+1},\ldots,Z_d)^T$ denote an $r$-dimensional real standard normal random vector and write 
$Z_\Sigma:=\Sigma^{1/2}Z$, where $\Sigma:= \on{diag}(d-r+1,d-r+2,\dots,d)$.
The objective of this section is to prove the following 

\begin{theorem} \label{thm-so}
If $n \ge 4d+1$, the Wasserstein distance between $W$ and $Z_\Sigma$ is of the same order as in the
unitary case, namely
\[d_\calw(W,Z_\Sigma)=\on{O}\left(\frac{\max\left\{\frac{r^{7/2}}{(d-r+1)^{3/2}},\,(d-r)^{3/2}\sqrt{r}\right\}}{n}\right)\,.\]
\end{theorem}

To prove Theorem \ref{thm-so}, we invoke a family $(W, W_t)_{t > 0}$ of exchangeable pairs such that 
 $W_t:=W_t(d,r,n):=(f_{d-r+1}(M_t),f_{d-r+2}(M_t),\ldots,f_d(M_t))^T$ for all $t > 0$. 
We will apply Prop. \ref{Meckes-r}, so the first step is to verify its conditions.
For condition (i) we will need the following 

\begin{lemma}\label{lemma2.4}
For all $j = d-r+1, \ldots, d$
\[\erw[W_{t,j}-W_j|M]=\erw[f_j(M_t)-f_j(M)|M]=t\cdot\left(-\frac{(n-1)j}{2}f_j(M)+R_j+ \on{O}(t)\right)\,,\]
where 
\begin{eqnarray*}
R_j &=& -\frac{j}{2}\sum_{l=1}^{j-1}p_{l,j-l}(M)+\frac{j}{2}\sum_{l=1}^{j-1}p_{2l-j}(M)\quad \text{\rm if}\ j\ 
\text{\rm is~odd,} \\
R_j &=& -\frac{(n-1)j}{2}-\frac{j}{2}\sum_{l=1}^{j-1}p_{l,j-l}(M)+\frac{j}{2}\sum_{l=1}^{j-1}p_{2l-j}(M)\quad
\text{\rm if}\ j\ \text{\rm is~even.}
\end{eqnarray*}
\end{lemma}

\begin{proof}
First observe that always $f_j(M_t)-f_j(M)=p_j(M_t)-p_j(M)$, no matter what the parity of $j$ is. By Lemmas \ref{entwicklung} and \ref{Rains-so}

\begin{eqnarray*}
\erw[p_j(M_t)-p_j(M)|M]&=&t(\Delta p_j)(M)+\on{O}(t^2)\\
&=&t\left(-\frac{(n-1)j}{2}p_j(M)-\frac{j}{2}\sum_{l=1}^{j-1}p_{l,j-l}(M)+\frac{j}{2}\sum_{l=1}^{j-1}p_{2l-j}(M)\right)+\on{O}(t^2)\,,
\end{eqnarray*}

which is 
\[t\Biggl(-\frac{(n-1)j}{2}f_j(M)+\Bigl(-\frac{j}{2}\sum_{l=1}^{j-1}p_{l,j-l}(M)+\frac{j}{2}\sum_{l=1}^{j-1}p_{2l-j}(M)\Bigr)+\on{O}(t)\Biggr)\]

if $j$ is odd and which is

\[t\Biggl(-\frac{(n-1)j}{2}f_j(M)+\Bigl(-\frac{(n-1)j}{2}-\frac{j}{2}\sum_{l=1}^{j-1}p_{l,j-l}(M)+\frac{j}{2}\sum_{l=1}^{j-1}p_{2l-j}(M)\Bigr)+\on{O}(t)\Biggr)\]

if $j$ is even. This proves the lemma.
\end{proof}

From Lemma \ref{lemma2.4} (see also Remark \ref{convtype}) we conclude

\[\frac{1}{t}\erw[W_t-W|M]\stackrel{t\to0}{\longrightarrow}-\Lambda W+R \text{ almost surely and in } \on{L}^1(\pp)\,,\]
 
where $\Lambda=\on{diag}\left(\frac{(n-1)j}{2}\,,\,j=d-r+1,\ldots,d\right)$ and $R=(R_{d-r+1},\ldots,R_d)^T$. Thus, condition (i) of Prop.\ \ref{Meckes-r} is satisfied. In order to verify condition (ii) we will first prove the following

\begin{lemma} \label{lemma2.5}
For all $j,k = d-r+1, \ldots, d$
$$\erw[(p_j(M_t)-p_j(M))(p_k(M_t)-p_k(M))|M]=t\left(jk p_{j-k}(M)-jk p_{j+k}(M)\right) + \on{O}(t^2).$$
\end{lemma}

\begin{proof}
By well-known properties of conditional expectation
\begin{eqnarray}\label{lemma2.5a}
&&\erw[(p_j(M_t)-p_j(M))(p_k(M_t)-p_k(M))|M]\nonumber\\ 
&=&\erw[p_{j,k}(M_t)|M]-p_j(M)\erw[p_k(M_t)|M]
-p_k(M)\erw[p_j(M_t)|M]+p_{j,k}(M)
\end{eqnarray}
 By Lemmas \ref{entwicklung} and  part (ii)  of Lemma \ref{Rains-so}, we see that 
\begin{eqnarray}\label{lemma2.5b}
&&\erw[p_{j,k}(M_t)|M]=p_{j,k}(M)+t\left(\Delta p_{j,k}\right)(M)+\on{O}(t^2)\nonumber\\
&=&p_{j,k}(M)+t\Bigl(-\frac{(n-1)(j+k)}{2}p_{j,k}(M)-\frac{j}{2}p_k(M)\sum_{l=1}^{j-1}p_{l,j-l}(M)
-\frac{k}{2}p_j(M)\sum_{l=1}^{k-1}p_{l,k-l}(M)\nonumber\\
&-&kj p_{j+k}(M) +\frac{j}{2}p_k(M)\sum_{l=1}^{j-1}p_{j-2l}(M)+ \frac{k}{2}p_j(M)\sum_{l=1}^{k-1}p_{k-2l}(M)
+jk p_{j-k}(M)\Bigr)\nonumber\\
&+&\on{O}(t^2)
\end{eqnarray}

Also, by Lemma \ref{entwicklung} and part (i) of Lemma \ref{Rains-so}

\begin{eqnarray}\label{lemma2.5c}
&&-p_j(M)\erw[p_k(M_t)|M]=-p_j(M)\left(p_k(M)+t\left(\Delta p_k\right)(M)+\on{O}(t^2)\right)\nonumber\\
&=&-p_{j,k}(M)-tp_j(M)\left(-\frac{(n-1)k}{2}p_k(M)-\frac{k}{2}\sum_{l=1}^{k-1}p_{l,k-l}(M)
+\frac{k}{2}\sum_{l=1}^{k-1} p_{2l-k}(M)\right)+\on{O}(t^2)\nonumber\\
&=&-p_{j,k}(M)+t\frac{(n-1)k}{2}p_{j,k}(M)+t\frac{k}{2}p_j(M)\sum_{l=1}^{k-1}p_{l,k-l}(M)\nonumber\\
&-&t\frac{k}{2}p_j(M)\sum_{l=1}^{k-1} p_{2l-k}(M) +\on{O}(t^2),
\end{eqnarray}
and for reasons of symmetry
\begin{eqnarray}\label{lemma2.5d}
&&-p_k(M)\erw[p_j(M_t)|M]=-p_{j,k}(M)+t\frac{(n-1)j}{2}p_{j,k}(M)+t\frac{j}{2}p_k(M)\sum_{l=1}^{j-1}p_{l,j-l}(M)\nonumber\\
&-&t\frac{j}{2}p_k(M)\sum_{l=1}^{j-1} p_{2l-j}(M) +\on{O}(t^2).
\end{eqnarray}

Plugging (\ref{lemma2.5b}), (\ref{lemma2.5c}) and (\ref{lemma2.5d}) into (\ref{lemma2.5a}) and noting that for $j\in\nn$

\[p_{-j}(M)=\on{Tr}(M^{-j})=\on{Tr}\left((M^T)^j\right)=\on{Tr}\left((M^j)^T\right)=\on{Tr}(M^{j})=p_j(M)\,,\]

we see that many terms cancel, and finally obtain

$$ \erw[(p_j(M_t)-p_j(M))(p_k(M_t)-p_k(M))|M]=t\left(jk p_{j-k}(M)-jk p_{j+k}(M)\right) + \on{O}(t^2)\,.$$
\end{proof}

Observing as above that, regardless of the parity of $j$, we have $f_j(M_t)-f_j(M)=p_j(M_t)-p_j(M)$, we can now easily compute
\begin{eqnarray*}
&&\frac{1}{t}\erw[(W_{t,j}-W_j)(W_{t,k}-W_k)|M]=\frac{1}{t}\erw[(f_j(M_t)-f_j(M))(f_k(M_t)-f_k(M))|M]\\
&=&\frac{1}{t}\erw[(p_j(M_t)-p_j(M))(p_k(M_t)-p_k(M))|M]=jk p_{j-k}(M)-jk p_{j+k}(M)+\on{O}(t^2)\\
&\stackrel{t\to0}{\rightarrow}&jk p_{j-k}(M)-jk p_{j+k}(M)\text{ a.s. and in } L^1(\pp)\,,
\end{eqnarray*}
for all $j,k= 1, \ldots,d$. Noting that for $j=k$ the last expression is $j^2n-j^2 p_{2j}(M)$ and that 
$2\Lambda\Sigma=\on{diag}((n-1)j^2\,,\,j=1,\ldots,d)$ we see that condition (ii) of Prop.\ \ref{Meckes-r} is satisfied with the matrix $S=(S_{j,k})_{j,k=1,\ldots, d}$ given by 

\[S_{j,k}=\begin{cases}
j^2(1-p_{2j}(M)),& j=k\\
jk p_{j-k}(M)-jk  p_{j+k}(M),& j\not= k\,.
\end{cases}\]

In order to show that condition (iii) of Prop.\ \ref{Meckes-r} holds, we will need the following facts:

\begin{lemma}\label{lemma2.6}
For all $j = 1, \ldots, d,\ n \ge 4d + 1$, 
\begin{enumerate}
\item[(i)] $\erw\left[(W_{t,j}-W_j)^2\right]=tj^2(n-1)+\on{O}(t^2).$
\item[(ii)] $\erw\left[(W_{t,j}-W_j)^4\right]=\on{O}(t^2).$
\end{enumerate}
\end{lemma}

\begin{proof}
As for (i), by Lemma \ref{lemma2.5},
\begin{eqnarray*} 
 &&\erw[(W_{t,j}-W_j)^2]=\erw\left[\left(p_j(M_t)-p_j(M)\right)^2\right]\\
&=&\erw\left[\erw\left[\left(p_j(M_t)-p_j(M)\right)^2|\ M\right]\right]
=\erw\left[tj^2\left(n-p_{2j}(M)\right)+\on{O}(t^2)\right]\\
&=&tj^2(n-1)+\on{O}(t^2)\,,
\end{eqnarray*}
since by Lemma \ref{DS-so} $\erw[p_{2j}(M)]=1$.\\
For claim (ii) we compute 
\begin{eqnarray*}
&&\erw\left[(W_{t,j}-W_j)^4\right]=\erw\left[\left(p_j(M_t)-p_j(M)\right)^4\right]\\
&=&\erw[p_j(M_t)^4]+\erw[p_j(M)^4]-4\erw[p_j(M_t)^3p_j(M)]-4\erw[p_j(M)^3p_j(M_t)]+6\erw[p_j(M_t)^2p_j(M)^2]\\
&=&2\erw[p_j(M)^4]-8\erw[p_j(M)^3p_j(M_t)]+6\erw[p_{j,j}(M)p_{j,j}(M_t)]\,,
\end{eqnarray*}

where the last equality follows from exchangeability. By Lemma \ref{entwicklung} and part (ii) of Lemma \ref{Rains-so} 
for the case $k=j$

\begin{eqnarray*}
&&\erw[p_{j,j}(M)p_{j,j}(M_t)]=\erw\left[p_{j,j}(M)\erw[p_{j,j}(M_t)|M]\right]\\
&=&\erw\left[p_{j,j}(M)\left(p_{j,j}(M)+t\left(\Delta p_{j,j}\right)(M)+\on{O}(t^2)\right)\right]\\
&=&\erw[p_j(M)^4]+t\erw[p_{j,j}(M)\cdot\Delta p_{j,j}(M)]+\on{O}(t^2)\\
&=&\erw[p_j(M)^4]+t\erw\Biggl[p_{j,j}(M)\biggl(-(n-1)jp_{j,j}(M)-jp_j(M)\sum_{l=1}^{j-1}p_{l,j-l}(M) -j^2p_{2j}(M)\\
&+&jp_j(M)\sum_{l=1}^{j-1}p_{j-2l}(M)+j^2n\biggr)\Biggr]+\on{O}(t^2)\,.
\end{eqnarray*}

Again by Lemma \ref{entwicklung} and part (i) of Lemma \ref{Rains-so},
\begin{eqnarray*}
&&\erw[p_j(M)^3p_j(M_t)]=\erw\left[p_j(M)^3\erw[p_j(M_t)|M]\right]\\
&=&\erw\left[p_j(M)^3\left((p_j(M)+t\left(\Delta p_j\right)(M)+\on{O}(t^2)\right)\right]\\
&=&\erw[p_j(M)^4]+t\erw\Biggl[p_j(M)^3\biggl(-\frac{(n-1)j}{2}p_j(M)-\frac{j}{2}\sum_{l=1}^{j-1}p_{l,j-l}(M)+
\frac{j}{2}\sum_{l=1}^{j-1}p_{2l-j}(M)\biggr)\Biggr] +\on{O}(t^2)\,.
\end{eqnarray*} 

Therefore, 
\begin{eqnarray*}
&&\erw\left[(W_{t,j}-W_j)^4\right]=2\erw\left[p_j(M)^4\right]-8\biggl(\erw[p_j(M)^4]-t\frac{(n-1)j}{2}\erw[p_j(M)^4]\\
&-&t\frac{j}{2}\sum_{l=1}^{j-1}\erw\left[p_j(M)^3 p_{l,j-l}(M)\right]
+t\frac{j}{2}\sum_{l=1}^{j-1}\erw\left[p_j(M)^3 p_{2l-j}(M)\right]\biggr)\\
&+&6\Biggl(\erw\left[p_j(M)^4\right]-t(n-1)j\erw\left[p_j(M)^4\right]-tj\sum_{l=1}^{j-1}\erw\left[p_j(M)^3 p_{l,j-l}(M)\right]\\
&-&tj^2\erw\left[p_{2j}(M)p_j(M)^2\right]+tj\sum_{l=1}^{j-1}\erw\left[p_j(M)^3 p_{2l-j}(M)\right]
+tj^2n\erw\left[p_j(M)^2\right]\Biggr)+\on{O}(t^2)\\
&=&-2t(n-1)j\erw\left[p_j(M)^4\right]-2tj\sum_{l=1}^{j-1}\erw\left[p_j(M)^3 p_{l,j-l}(M)\right]
+2tj\sum_{l=1}^{j-1}\erw\left[p_j(M)^3 p_{2l-j}(M)\right]\\
&-&6tj^2\erw\left[p_{2j}(M)p_j(M)^2\right]+6tj^2n\erw\left[p_j(M)^2\right]+\on{O}(t^2).
\end{eqnarray*}

\underline{Case~1: $j$ is odd.} Then by Lemma \ref{DS-so}

\begin{eqnarray*}
\erw\left[(W_{t,j}-W_j)^4\right]&=&-2t(n-1)j\cdot3j^2-0+0-6tj^2\cdot j+6tj^2n\cdot j+\on{O}(t^2)\\
&=&\on{O}(t^2)\,,
\end{eqnarray*}

as claimed.\\

\underline{Case~2: $j$ is even.} Then, again by Lemma \ref{DS-so}
\begin{eqnarray*}
&&\erw\left[(W_{t,j}-W_j)^4\right]=-2t(n-1)j(1+6j+3j^2)-2tj(1+3j)\sum_{l=1}^{j-1}\erw\left[p_{l,j-l}(M)\right]\\
&+&2tj(1+3j)\sum_{l=1}^{j-1}\erw\left[p_{2l-j}(M)\right]-6tj^2(1+j)+6tj^2n(1+j)+\on{O}(t^2)\,.
\end{eqnarray*}

Consider the term $\sum_{l=1}^{j-1}\erw\left[p_{l,j-l}(M)\right]$. Since $\erw\left[p_{l,j-l}(M)\right]=0$ whenever $l$ is odd and $l\not=\frac{j}{2}$, we can write

\begin{eqnarray*}
\sum_{l=1}^{j-1}\erw\left[p_{l,j-l}(M)\right]&=&\erw\left[p_{j/2,j/2}(M)\right]+
\sum_{\substack{k=1\\k\not=j/4}}^{j/2-1}\erw\left[p_{2k,j-2k}(M)\right]\\
&=&\erw\left[p_{j/2,j/2}(M)\right]+\sum_{\substack{k=1\\k\not=j/4}}^{j/2-1}1\,,
\end{eqnarray*}

where the last equality follows again by Lemma \ref{DS-so}. If $j/2$ is odd, then, clearly, $k\not=j/4$ is not really a restriction, and by Lemma \ref{DS-so}

\[\sum_{l=1}^{j-1}\erw\left[p_{l,j-l}(M)\right]=\frac{j}{2}+\left(\frac{j}{2}-1\right)=j-1\,.\]

If $j/2$ is even, then by Lemma \ref{DS-so}

\[\sum_{l=1}^{j-1}\erw\left[p_{l,j-l}(M)\right]=\left(1+\frac{j}{2}\right)+\left(\frac{j}{2}-2\right)=j-1\,.\]

Hence, in either case we have $\sum_{l=1}^{j-1}\erw\left[p_{l,j-l}(M)\right]=j-1$.\\

Furthermore 

\[\sum_{l=1}^{j-1}\erw\left[p_{2l-j}(M)\right]=2\sum_{l=1}^{j/2-1}\erw\left[p_{2l-j}(M)\right]+\erw\left[\on{Tr}(I_n)\right]
=2\left(\frac{j}{2}-1\right)+n=j-2+n\,.\]

Therefore, if $j$ is even

\begin{eqnarray*} 
&&\erw\left[(W_{t,j}-W_j)^4\right]=-2tj(n-1)(1+6j+3j^2)-2tj(1+3j)(j-1)\\
&+&2tj(1+3j)(j-2+n)-6tj^2(1+j)+6tj^2n(1+j)+\on{O}(t^2)\\
&=&tj\left(-12(n-1)-2+6+2+6(n-2)-6+6n\right)\\
&+&tj^3\left(-6(n-1)-6+6-6+6n\right)+\on{O}(t^2)\\
&=&\on{O}(t^2)\,,
\end{eqnarray*}

as asserted.
\end{proof}

Now we are in a position to check condition $(iii)\tra$ of Prop.\ \ref{Meckes-r}. By H\"{o}lder's inequality 

\begin{eqnarray*}
\erw\left[\|W_t-W\|_2^3\right]&=&\erw\left[\left(\sum_{j=d-r+1}^d(W_{t,j}-W_j)^2\right)^{3/2}\right]\\
&=&\erw\left[\left(\sum_{j,k,l=d-r+1}^d(W_{t,j}-W_j)^2 (W_{t,k}-W_k)^2 (W_{t,l}-W_l)^2\right)^{1/2}\right]\\
&\leq&\sum_{j,k,l=d-r+1}^d\erw\left[\left|(W_{t,j}-W_j) (W_{t,k}-W_k) (W_{t,l}-W_l)\right|\right]\\
&\leq&\sum_{j,k,l=d-r+1}^d\biggl(\erw\left[|W_{t,j}-W_j|^3\right]\erw\left[|W_{t,k}-W_k|^3\right]\erw\left[|W_{t,l}-W_l|^3\right]\biggr)^{1/3}\,.
\end{eqnarray*}

Thus we have 

\begin{eqnarray*}
\frac{1}{t}\erw\left[\|W_t-W\|_2^3\right]\leq\sum_{j,k,l=d-r+1}^d\biggl(\frac{1}{t}\erw\left[|W_{t,j}-W_j|^3\right]\frac{1}{t}\erw\left[|W_{t,k}-W_k|^3\right]\frac{1}{t}\erw\left[|W_{t,l}-W_l|^3\right]\biggr)^{1/3},
\end{eqnarray*}

and it suffices to show that for all $j = d-r+1, \ldots, d$ 
\[\lim_{t\to0}\frac{1}{t}\erw\left[|W_{t,j}-W_j|^3\right]=0\,.\]
But this follows from Lemma \ref{lemma2.6}, since by the Cauchy-Schwarz inequality 
\begin{eqnarray*}
\erw\left[|W_{t,j}-W_j|^3\right]&\leq&\sqrt{\erw\left[|W_{t,j}-W_j|^2\right]\erw\left[|W_{t,j}-W_j|^4\right]}
=\sqrt{\left(tj^2(n-1)+\on{O}(t^2)\right)\cdot \on{O}(t^2)}\\
&=&\sqrt{\on{O}(t^3)}=\on{O}(t^{3/2})
\end{eqnarray*}
and hence $\frac{1}{t}\erw\left[|W_{t,j}-W_j|^3\right]\stackrel{t\to0}{\rightarrow}0$.\\

By Prop.\ \ref{Meckes-r} we can now conclude that
\[d_\calw(W,Z_\Sigma)\leq \|\Lambda^{-1}\|_{\rm op}
\left(\erw[\|R\|_2]+\frac{1}{\sqrt{2\pi}}\|\Sigma^{-1/2}\|_{\rm op}\erw[\|S \|_{\rm HS}]\right)\,.\]

Clearly, $\|\Lambda^{-1}\|_{\rm op}=\frac{2}{(n-1)(d-r+1)}$ and $\|\Sigma^{-1/2}\|_{\rm op}=\frac{1}{\sqrt{d-r+1}}$. 
In order to bound $\erw[\|R\|_2]$ we will first prove the following lemma.

\begin{lemma}\label{lemma2.7}
If $n \ge 4d+1$, for all $j= d-r+1, \ldots, d$ we have that $\erw[R_j^2]=\on{O}(j^5)$.
\end{lemma}

\begin{proof}
First suppose that $j$ is odd. Then 

\begin{eqnarray*}
&&\erw[R_j^2]=\erw\left[\left(-\frac{j}{2}\sum_{l=1}^{j-1}p_{l,j-l}(M)+\frac{j}{2}\sum_{l=1}^{j-1}p_{2l-j}(M)\right)^2\right]\\
&=&\frac{j^2}{4}\left(\sum_{l,k=1}^{j-1}\erw[p_{l,j-l}(M)p_{k,j-k}(M)]-2\sum_{l,k=1}^{j-1}\erw[p_{l,j-l}(M)p_{2k-j}(M)]
+\sum_{l,k=1}^{j-1}\erw[p_{2l-j}(M)p_{2k-j}(M)]\right)\\
&=:&\frac{j^2}{4}\left(T_1-2T_2+T_3\right)\,.
\end{eqnarray*} 

By Lemma \ref{DS-so} for $j$ odd

\begin{eqnarray*}
\erw[p_{l,j-l}(M)p_{k,j-k}(M)]&=&
\begin{cases}
l(j-l+1),& \text{if }k=l \text{ is odd,}\\
(l+1)(j-l),&\text{if }k=l \text{ is even,}\\
l(j-l+1),&\text{if }l \text{ is odd and } k=j-l,\\
0, &\text{if }l \text{ is odd and } k\notin\{l,j-l\},\\
(l+1)(j-l), &\text{if }l \text{ is even and } k=j-l,\\
0, &\text{if }l \text{ is even and }k\notin\{l,j-l\},
\end{cases}
\end{eqnarray*}
\begin{eqnarray*}
\erw[p_{l,j-l}(M)p_{2k-j}(M)]&=&
\begin{cases}
l, &\text{if }l=2k-j, \\
j-l, &\text{if }l=2j-2k,\\
l, &\text{if }l=j-2k,\\
j-l, &\text{if }l=2k,\\
0, &\text{otherwise,}
\end{cases}\\
\erw[p_{2l-j}(M)p_{2k-j}(M)]&=&
\begin{cases}
|2l-j|,& \text{if }l=k\text{ or }l=j-k,\\
0,& \text{otherwise.}
\end{cases}
\end{eqnarray*}

Therefore, 

\begin{eqnarray*}
T_1&=&2\sum_{\substack{l=1\\l\text{ odd}}}^{j-1}l(j-l+1)+2\sum_{\substack{l=1\\l\text{ even}}}^{j-1}(l+1)(j-l)\\
&\leq&2j\sum_{l=1}^j l= j^2(j-1)=\on{O}(j^3)\,,
\end{eqnarray*}

\begin{eqnarray*}
T_2&=&2\sum_{\substack{l=1\\l\text{ odd}}}^{j-1}l+2\sum_{\substack{l=1\\l\text{ even}}}^{j-1}(j-l)
\leq2\sum_{l=1}^jl=j(j-1)=\on{O}(j^2)\,
\end{eqnarray*}
and
\[T_3=2\sum_{l=1}^{j-1}|2l-j|\leq2\sum_{l=1}^j j\ = 2j^2=\on{O}(j^2)\,.\]

Hence for $j$ odd

\[\erw[R_j^2]=\frac{j^2}{4}\left(\on{O}(j^3)-2\on{O}(j^2)+\on{O}(j^2)\right)=\on{O}(j^5)\,.\]

The case that $j$ is even can be treated similarly as can be seen by observing that in this case 

\[R_j=\frac{j}{2}\left(1-\sum_{l=1}^{j-1}p_{l,j-l}(M)+\sum_{\substack{l=1\\l\not=j/2}}^{j-1}p_{2l-j}\right)\,.\]

But as in Case 2 of the proof of Lemma \ref{lemma2.6} (ii), one has to distinguish between the cases that 
$j\equiv0 \on{mod}4$ or else that $j\equiv2 \on{mod}4$, and the explicit formulae for $\erw[R_j^2]$ are more complicated.
\end{proof}

By Lemma \ref{lemma2.7} there is a constant $C>0$ neither depending on $j$ nor on $n$ such that $\erw[R_j^2]\leq Cj^5$.
Thus by Jensen's inequality and the binomial theorem we obtain

\begin{eqnarray} \label{j5}
&&\erw[\|R\|_2]\leq\sqrt{\sum_{j=d-r+1}^d \erw[R_j^2]}\leq\sqrt{\sum_{j=d-r+1}^d Cj^5}=\sqrt{C}\sqrt{\sum_{j=1}^r (j+d-r)^5}\nonumber\\
&=&\sqrt{C}\left(\sum_{j=1}^r\left(j^5+5j^4(d-r)+10j^3(d-r)^2+10j^2(d-r)^3+5j(d-r)^4+(d-r)^5\right)\right)^{1/2}\nonumber\\
&\leq&\sqrt{C}\left(r^6+5r^5(d-r)+10r^4(d-r)^2+10r^3(d-r)^3+5r^2(d-r)^4+r(d-r)^5\right)^{1/2}\nonumber\\
&\leq&\sqrt{C}\sqrt{32}\left(\max\{r^6, (d-r)^5r\}\right)^{1/2}=O\left(\max\{r^3,(d-r)^{5/2}\sqrt{r}\}\right)\,.
\end{eqnarray}

Next, we turn to bounding $\erw[\|S \|_{\rm HS}]$. By Lemma \ref{DS-so} we have

\begin{eqnarray*}
&&\erw[\|S \|_{\rm HS}^2]=\sum_{j,k=d-r+1}^d\erw[S_{j,k}^2]=\sum_{j=d-r+1}^d j^4\erw\left[1-2p_{2j}(M)+p_{2j}(M)^2\right]\\
&+&\sum_{\substack{j,k=d-r+1\\j\not=k}}^d j^2k^2\erw\left[p_{j-k}(M)^2-2p_{j-k}(M)p_{j+k}(M)+p_{j+k}(M)^2\right]\\
&=&2\sum_{j=d-r+1}^d j^5+2\sum_{d-r+1\leq k<j\leq d}j^2k^2\left(\erw\left[p_{j-k}(M)^2\right]-2\erw\left[p_{j-k}(M)p_{j+k}(M)\right]+\erw\left[p_{j+k}(M)^2\right]\right)\,.
\end{eqnarray*}

Again, by Lemma \ref{DS-so}

\begin{eqnarray*}
\erw\left[p_{j-k}(M)^2\right]&=&
\begin{cases}
1+j-k,&\text{if }j+k \text{ even,}\\
j-k,&\text{if }j+k\text{ odd,}
\end{cases}\\
\erw\left[p_{j-k}(M)p_{j+k}(M)\right]&=&
\begin{cases}
1,~~~~~~&\text{if } j+k\text{ even,}\\
0,&\text{if } j+k\text{ odd,}
\end{cases}\\
\erw\left[p_{j+k}(M)^2\right]&=&
\begin{cases}
1+j+k,&\text{if } j+k\text{ even}\\
j+k,&\text{if } j+k\text{ odd.}
\end{cases}
\end{eqnarray*}

Hence, 

\begin{eqnarray*}
\erw[\|S \|_{\rm HS}^2]
&=&2\sum_{j=d-r+1}^d j^5+2\sum_{\substack{d-r+1\leq k<j\leq d\\k+j\text{ even}}}j^2k^2\left(1+j-k-2+1+j+k\right)\\
&+&2\sum_{\substack{d-r+1\leq k<j\leq d\\k+j\text{ odd}}}j^2k^2\left(j-k-2\cdot0+j+k\right)=2\sum_{j=d-r+1}^d j^5+4\sum_{d-r+1\leq k<j\leq d}k^2j^3\,.
\end{eqnarray*}

From \eqref{j5} we know that 

\[\sum_{j=d-r+1}^d j^5=\sum_{j=1}^r(j+d-r)^5\leq 32\max\{r^6,\, (d-r)^5r\},\]

and furthermore we can bound

\begin{eqnarray*}
&&\sum_{d-r+1\leq k<j\leq d}k^2j^3\leq\sum_{j=d-r+2}^dj^2(j-(d-r+1))j^3\\
&=&\sum_{j=2}^r(j+d-r)^2(j+d-r-(d-r+1))(j+d-r)^3=\sum_{j=2}^r(j+d-r)^5(j-1)\\
&\leq&r\sum_{j=1}^r(j+d-r)^5\leq 32r\max\{r^6,\, (d-r)^5r\}=32\max\{r^7,\,(d-r)^5r^2\}\,.
\end{eqnarray*}

Thus, we obtain 

\begin{eqnarray*}
\erw[\|S \|_{\rm HS}^2]&\leq& 64\max\{r^6,\,(d-r)^5r\} + 128\max\{r^7,\,(d-r)^5r^2\}\\
&=&O\left(\max\{r^7,\,(d-r)^5r^2\}\right)\,,
\end{eqnarray*}

and hence, by Jensen's inequality

\[\erw[\|S \|_{\rm HS}]\leq\sqrt{\erw[\|S \|_{\rm HS}^2]}=O\left(\max\{r^{7/2},\,(d-r)^{5/2}r\}\right)\,,\]

Collecting terms, we see that 
\begin{eqnarray*}
d_\calw(W,Z_\Sigma)&\leq&\|\Lambda^{-1}\|_{\rm op}\erw[\|R\|_2]+\frac{1}{\sqrt{2\pi}}\|\Lambda^{-1}\|_{\rm op}\|\Sigma^{-1/2}\|_{\rm op}\erw[\|S \|_{\rm HS}]\\
&=&\frac{2}{(n-1)(d-r+1)}O\left(\max\{r^3,(d-r)^{5/2}\sqrt{r}\}\right)\\
&+&\frac{1}{\sqrt{2\pi}}\frac{2}{(n-1)(d-r+1)^{3/2}}O\left(\max\{r^{7/2},\,(d-r)^{5/2}r\}\right)\\
&=&O\left(\frac{\max\left\{\frac{r^3}{d-r+1},\, \frac{r^{7/2}}{(d-r+1)^{3/2}},\, \frac{(d-r)^{5/2}\sqrt{r}}{d-r+1},\, \frac{(d-r)^{5/2}r}{(d-r+1)^{3/2}}\right\}}{n}\right)\,.
\end{eqnarray*}
Proceeding as in the unitary case and treating the cases that $r \le d-r$ and $d-r < r$ separately, one obtains that the last
expression is of order
$$\frac{1}{n}\max\left\{\frac{r^{7/2}}{(d-r+1)^{3/2}},\,(d-r)^{3/2}\sqrt{r}\right\}.$$
This concludes the proof of Theorem \ref{thm-so}.

\section{The symplectic group}
\label{sec-sp}
Let $M = M_n$ be distributed according to Haar measure on $K_n = \on{USp}_{2n}$ and let $d \in \nn,\ r = 1, \ldots, d$. Consider the random vector
$W:=W(d,r,n):=(f_{d-r+1}(M),f_{d-r+2}(M),\ldots,f_d(M))^T$, where

\[f_j: \on{USp}_{2n}\rightarrow\rr\,,\quad f_j=
\begin{cases}
p_j,&  j\text{ odd}\\
p_j+1,& j\text{ even.}
\end{cases}\]

Let $Z=(Z_{d-r+1},\ldots,Z_d)^T$ denote an $r$-dimensional real standard normal random vector,
$\Sigma:=\on{diag}(d-r+1,d-r+2,\dots,d)$, and write 
$Z_\Sigma:=\Sigma^{1/2}Z$.

The objective of this section is to prove the following 

\begin{theorem} \label{thm-sp}
\item If $n \ge 2d$, the Wasserstein distance between $W$ and $Z_\Sigma$ is again
of the same order as in the unitary case, namely, as given in \eqref{ordnungsformel}.
\end{theorem}

For the proof, we use again exchangeable pairs $(W, W_t)_{t>0}$ such that 
$$W_t =(f_{d-r+1}(M_t),f_{d-r+2}(M_t),\ldots,f_d(M_t))^T,$$
and apply Proposition \ref{Meckes-r}. The verification of condition (i) involves the following analog
of Lemma  \ref{lemma2.4} above.

\begin{lemma}\label{lemma3.4}
For all $j = d-r+1, \ldots, d$,
\[\erw[W_{t,j}-W_j|M]=\erw[f_j(M_t)-f_j(M)|M]=t\cdot\left(-\frac{(2n+1)j}{2}f_j(M)+R_j+\on{O}(t)\right)\,,\]
where 
\begin{eqnarray*}
R_j &=& -\frac{j}{2}\sum_{l=1}^{j-1}p_{l,j-l}(M)-\frac{j}{2}\sum_{l=1}^{j-1}p_{2l-j}(M)\text{ if } j \text{ is odd}, \\
R_j &=& \frac{(2n+1)j}{2}-\frac{j}{2}\sum_{l=1}^{j-1}p_{l,j-l}(M)-\frac{j}{2}\sum_{l=1}^{j-1}p_{2l-j}(M)\text{ if }j
\text{ is even.}
\end{eqnarray*}
\end{lemma}

\begin{proof} 
First observe that always $f_j(M_t)-f_j(M)=p_j(M_t)-p_j(M)$, no matter what the parity of $j$ is. By Lemmas \ref{entwicklung} and \ref{Rains-sp},

\begin{eqnarray*}
\erw[p_j(M_t)-p_j(M)|M]&=&t(\Delta p_j)(M)+\on{O}(t^2)\\
&=&t\left(-\frac{(2n+1)j}{2}p_j(M)-\frac{j}{2}\sum_{l=1}^{j-1}p_{l,j-l}(M)-\frac{j}{2}\sum_{l=1}^{j-1}p_{2l-j}(M)\right)+\on{O}(t^2)\,,
\end{eqnarray*}

which is 
\[t\Biggl(-\frac{(2n+1)j}{2}f_j(M)+\Bigl(-\frac{j}{2}\sum_{l=1}^{j-1}p_{l,j-l}(M)-\frac{j}{2}\sum_{l=1}^{j-1}p_{2l-j}(M)\Bigr)+\on{O}(t)\Biggr)\]

if $j$ is odd and which is

\[t\Biggl(-\frac{(2n+1)j}{2}f_j(M)+\Bigl(\frac{(2n+1)j}{2}-\frac{j}{2}\sum_{l=1}^{j-1}p_{l,j-l}(M)-\frac{j}{2}\sum_{l=1}^{j-1}p_{2l-j}(M)\Bigr)+\on{O}(t)\Biggr)\]

if $j$ is even. This proves the lemma.
\end{proof}

From Lemma \ref{lemma3.4} we conclude that 

\[\frac{1}{t}\erw[W_t-W|M]\stackrel{t\to0}{\longrightarrow}-\Lambda W+R \text{ a.s.\  and in }\on{L}^1(\pp)\,,\]
 
where $\Lambda=\on{diag}\left(\frac{(2n+1)j}{2}\,,\,j=d-r+1,\ldots,d\right)$ and $R=(R_{d-r+1},\ldots,R_d)^T$. Thus, condition (i) of Prop.\ \ref{Meckes-r} is satisfied. The validity of condition (ii) follows from the next lemma, whose proof only differs from that of Lemma \ref{lemma2.5} in that it makes use of Lemma \ref{Rains-sp} in the place of Lemma \ref{Rains-so}.

\begin{lemma} \label{lemma3.5}
For all $j,k = d-r+1, \ldots, d,$ 

\[\erw[(p_j(M_t)-p_j(M))(p_k(M_t)-p_k(M))|M]=t\left(jk p_{j-k}(M)-jk p_{j+k}(M)\right) + \on{O}(t^2)\,.\]

\end{lemma}

Observing that for all $j = d-r+1, \ldots, d$ we have $f_j(M_t)-f_j(M)=p_j(M_t)-p_j(M)$ and using Lemma \ref{lemma3.5} we can now easily compute

\begin{eqnarray*}
&&\frac{1}{t}\erw[(W_{t,j}-W_j)(W_{t,k}-W_k)|M]=\frac{1}{t}\erw[(f_j(M_t)-f_j(M))(f_k(M_t)-f_k(M))|M]\\
&=&\frac{1}{t}\erw[(p_j(M_t)-p_j(M))(p_k(M_t)-p_k(M))|M]=jk p_{j-k}(M)-jk p_{j+k}(M)+\on{O}(t^2)\\
&\stackrel{t\to0}{\rightarrow}&jk p_{j-k}(M)-jk p_{j+k}(M)\text{ a.s. and in } \on{L}^1(\pp)\,,
\end{eqnarray*}

for all $d-r+1\leq j,k\leq d$. Noting that for $j=k$ the last expression is $2j^2n-j^2 p_{2j}(M)$ and that 
$2\Lambda\Sigma=\on{diag}((2n+1)j^2\,,\,j=d-r+1,\ldots,d)$, we see that condition (ii) of Prop.\ \ref{Meckes-r} is satisfied with the matrix $S=(S_{j,k})_{j,k=d-r+1,\ldots, d}$ given by 

\[S_{j,k}=\begin{cases}
-j^2(1+p_{2j}(M)),& j=k\\
jk p_{j-k}(M)-jk  p_{j+k}(M),& j\not= k\,.
\end{cases}\]
The validity of condition $(iii\tra)$ of Prop.\ \ref{Meckes-r} is based on the following lemma, which can be 
proven in the same way as its analog in Section \ref{sec-so}.

\begin{lemma}\label{lemma3.6}
If $n \ge 2d$, for all $j = 1, \ldots, d$ there holds
\begin{enumerate}
\item[(i)] $\erw\left[(W_{t,j}-W_j)^2\right]=tj^2(2n+1)+\on{O}(t^2).$
\item[(ii)] $\erw\left[(W_{t,j}-W_j)^4\right]=\on{O}(t^2).$
\end{enumerate}
\end{lemma}
So we obtain from Prop.\ \ref{Meckes-r} that

\[d_\calw(W,Z_\Sigma)\leq \|\Lambda^{-1}\|_{\rm op}
\left(\erw[\|R\|_2]+\frac{1}{\sqrt{2\pi}}\|\Sigma^{-1/2}\|_{\rm op}\erw[\|S \|_{\rm HS}]\right)\,.\]

Again, it is easy to see that $\|\Lambda^{-1}\|_{\rm op}=\frac{2}{(2n+1)(d-r+1)}$ and $\|\Sigma^{-1/2}\|_{\rm op}=\frac{1}{\sqrt{d-r+1}}$. Since we can show in a similar way as in Section \ref{sec-so} that $\erw[\|R\|_2]=O\left(\max\{r^3,(d-r)^{5/2}\sqrt{r}\}\right)$ and $\erw[\|S \|_{\rm HS}]=O\left(\max\{r^{7/2},\,(d-r)^{5/2}r\}\right)$, we may conclude the proof of Theorem \ref{thm-sp} as before.


\end{document}